\documentclass[10pt]{amsart}

\usepackage{amsmath,amssymb,graphicx,bbm}
\usepackage{amsthm,verbatim}
\usepackage{mathrsfs,mathtools}
\usepackage{float}

\usepackage[footnotesize,bf]{caption}
\usepackage[left=1.15in,right=1.15in,top=1in]{geometry}
\usepackage[dvipsnames]{xcolor}

\usepackage{mathdefs}

\newcommand{\calN}{\mathcal{N}}
\newcommand{\bs}[1]{\boldsymbol{#1}}
\DeclareMathOperator{\sinc}{sinc}

\newcommand{\naive}{naive}

\usepackage{algorithm}
\usepackage{algorithmicx}
\usepackage[noend]{algpseudocode}

\algrenewcommand\algorithmicrequire{\textbf{Precondition:}}
\algrenewcommand\algorithmicensure{\textbf{Postcondition:}}

\title{Model reduction for fractional elliptic problems using Kato's formula}

\thanks{H.~Antil was partially supported by National Science Foundation awards DMS-1818772 and DMS-1521590 and Air Force Office of Scientific Research (AFOSR) under Award NO: FA9550-19-1-0036. Y.~Chen's research was supported by NSF awards DMS-1719698 and AFOSR award FA9550-18-1-0383. E.~Cherkaev was supported by NSF awards DMS-0940249 and DMS-1413454. A.~Narayan was partially supported by NSF award DMS-1720416, AFOSR award FA9550-15-1-0467, and DARPA EQUiPS contract N660011524053.}

\author{Huy Dinh \and Harbir Antil \and Yanlai Chen \and Elena Cherkaev \and Akil Narayan}

\begin{document}
\maketitle

\begin{abstract}
  We propose a novel numerical algorithm utilizing model reduction for computing solutions to stationary partial differential equations involving the spectral fractional Laplacian. Our approach utilizes a known characterization of the solution in terms of an integral of solutions to classical elliptic problems. We reformulate this integral into an expression whose continuous and discrete formulations are stable; the discrete formulations are stable independent of all discretization parameters. We subsequently apply the reduced basis method to accomplish model order reduction for the integrand. Our choice of quadrature in discretization of the integral is a global Gaussian quadrature rule that we observe is more efficient than previously proposed quadrature rules. Finally, the model reduction approach enables one to compute solutions to multi-query fractional Laplace problems with order of magnitude less cost than a traditional solver.
\end{abstract}

\section{Introduction}
Differential equations involving fractional derivative powers have gained in popularity in recent years. 
These non-classical differential equations have shown potential to model nonlocal and time-delay effects, making them good candidates for modeling hysteric and globally-coupled phenomena. 
For example, fractional differential equations have recently been used to model fluid mechanics, arterial blood flow, cardiac ischemia, and are been used as ingredients in image denoising and image segmentation \cite{perdikaris_fractional-order_2014,kumar_fractional_2015,antil_sobolev_2018,farquhar_computational_2018,HAntil_SBartels_2017a,HAntil_SBartels_GDogan_2019a}. Fractional PDEs have shown tremendous potential to model applications in geophysics \cite{CWeiss_BvBWaanders_HAntil_2018a} and manifold learning \cite{antil2018fractional}. Finally, we mention the novel optimal control concepts introduced by fractional equations \cite{HAntil_RKhatri_MWarma_2018a}, see also \cite{HAntil_MWarma_2018a, HAntil_MWarma_2018b, HAntil_MWarma_2019a}.
In this paper we focus on fractional elliptic operators, defined via spectral expansion; a prototypical example of an operator that we use throughout this paper is the fractional Laplacian.

With $\Delta$ the classical Laplacian on a physical domain $\Omega$, we are interested in computing the solution $u$ to the partial differential equation
\begin{align*}
  (-\Delta)^s u = f,\hskip 20pt & \hskip 20pt x \in \Omega, 
\end{align*}
with appropriate boundary conditions on $\partial \Omega$, where $s \in (0,1)$ is the fractional order. 
The precise definition of the fractional operator $(-\Delta)^s$ involves the spectral expansion of the classical operator $-\Delta$, which we more formally describe in Section \ref{sec:notation}. 
There are already several numerical algorithms for computing the solution to such an equation: 
\begin{itemize}
  \item Perhaps the most conceptually straightforward idea is to use the spectral expansion definition of $(-\Delta)^s$ to devise a scheme that computes solutions using the spectral expansion of the associated discretized operators \cite{ilic_numerical_2005,ilic_numerical_2006,yang_novel_2011,song_computing_2017}. 
The disadvantage of this approach is that the procedure is expensive, requiring a full eigendecomposition of a potentially very large matrix. In addition, it is difficult to certify error using this approach.
  \item A second approach uses an extension procedure to write the non-classical $d$-dimensional PDE as a $(d+1)$-dimensional classical PDE \cite{molchanov_symmetric_1969,caffarelli_extension_2007,stinga_extension_2010}. This latter PDE can be solved with existing methods, although some nontrivial tailoring of existing numerical methods is needed \cite{nochetto_pde_2014,meidner_$hp$-finite_2017,ainsworth_hybrid_2018}. The challenge with this approach is that the spatial dimension is increased, and the extended PDE is degenerate, requiring specialized numerical methods.
  \item A final approach that we use as the starting point for the method proposed in this paper is an integral operator approach, which writes the solution as a type of Dunford-Taylor integral involving the resolvent of the classical operator see \cite[Theorem 2 with simplification $\lambda = 0$]{Kato_60}, see also \cite{bonito_numerical_2015,antil_short_2017}. 
    This approach discretizes the integral formulation with quadrature and requires several classical PDE solves (equal to the number of quadrature points) in order to compute the solution to the fractional problem. However, this results in an algorithm that can use several queries of an existing PDE solver to compute the solution to the fractional problem. The challenge with this approach is that $\mathcal{O}(100)$ classical PDE solves may be necessary to ensure accuracy for a single solution of the fractional PDE, making this approach quite expensive compared to traditional solvers. For certain operators these classical PDE solves may be accomplished in parallel, but this does not diminish the overall cost. Other operators require coupling of these solves, making parallel approaches more difficult \cite{weiss_fractional_2019}.
\end{itemize}
In this paper, we develop a novel model reduction algorithm for the third approach listed above to substantially alleviate the cost of traditional PDE solves. We concentrate on this approach for the spectral definition of the fractional Laplacian in this paper, but due to a similar integral formulation for the integral fractional Laplacian \cite{bonito_numerical_2019}, our approach would extend to more general cases as well. Our contributions in this article are as follows:
\begin{itemize}
  \item We provide a rearrangement of the Dunford-Taylor integral considered in \cite{bonito_numerical_2015} that improves numerical stability. We first show that the analytical solution has $s$-independent $L^2$ stability bounds. This stability extends to the numerical discretization, independent of all discretization parameters. See Lemma \ref{lemma:u-partition} and Proposition \ref{prop:u-stability}. 
  \item Our approach to discretize the Dunford-Taylor integral is a novel application of a global Gaussian quadrature rule. 
    Our numerical results suggest that our quadrature choice is more efficient than previously proposed choices, cf. Figure \ref{fig:quadplot}. 
    We cannot provide an analytical error bound in terms of the number of quadrature points, but we do provide a rigorous, computable error certificate; see Proposition \ref{prop:qerror}.  
    Previous work has required a number of quadrature points proportional to $\max(1/s,1/(1-s))$ in order to obtain a specified level of accuracy. 
    Our empirical results suggest that our approach also suffers from this limitation Figure \ref{fig:Mvals}.
  \item We employ the reduced basis method (RBM) to effect model reduction which, after a single \textit{offline} computational investment, can accelerate subsequent computations of $(s,f) \mapsto u$ by at least two orders of magnitude. The offline portion of this algorithm requires approximately as much time as a single $(s,f) \mapsto u$ solve using the traditional Dunford-Taylor approach; see Algorithm \ref{ssec:gq:algorithm}.
  \item We provide a rigorous \textit{a posteriori} error estimate for our solution computed via model reduction. This error estimate is computed as a by-product of the offline investment, and is therefore directly available; see Theorem \ref{thm:u-certificate}. 
\end{itemize}
We remark that while we study PDEs with an operator of the form $(-\Delta)^s$, all our results extend to more general fractional elliptic operators. See Remarks \ref{rem:general-a} and \ref{rem:general-b}.

This paper is not the first strategy for model reduction for fractional elliptic problems. The authors in \cite{antil_certified_2018} provide a model reduction strategy, applied to the second (extension) approach listed above. More recently, the work in \cite{danczul_reduced_2019} employs a reduced basis approach by interpolating operator norms. However, low-rank structure in solution sets to fractional problems has been empirically noted even earlier \cite{witman_reduced-order_2016}. For problems involving nonlocal integral kernels, the authors in \cite{guan_reduced_2017} also proposed a reduced basis approach, but use a different strategy to perform model reduction.

This paper is structured as follows. Section \ref{sec:notation} lays out our notation and describes the problem. Section \ref{sec:gq} describes a new algorithm for expressing and computing the Dunford-Taylor solution that was first proposed in \cite{bonito_numerical_2015}. Section \ref{sec:rbm} utilizes RBM to propose a new model reduction algorithm that computationally accelerates the algorithm from Section \ref{sec:gq} and provides a computable error certificate for the model reduction. Finally, section \ref{sec:results} demonstrates our new algorithms on a two-dimensional fractional Laplace problem and compares our algorithm against the predecessor in \cite{bonito_numerical_2015}.

\section{Notation and setup}\label{sec:notation}
Vectors will be denoted in lowercase bold, and matrices in uppercase bold, e.g., $\bs{x}$ and $\bs{A}$, respectively. If $\bs{M}$ is a symmetric positive definite matrix, we define
\begin{align*}
  \| \bs{x} \|^2_{\bs{M}} &\coloneqq \bs{x}^T \bs{M} \bs{x}, 
\end{align*}
and $\|\bs{x}\|$ is the standard Euclidean norm. The matrix norm $\|\bs{A}\|$ is the standard induced $\ell^2$ norm on matrices. If $\bs{A}$ is symmetric, then $\lambda_{\mathrm{\min}}\left(\bs{A}\right)$ denotes the smallest (real) eigenvalue of $\bs{A}$. If both $\bs{A}$ and $\bs{B}$ are symmetric positive definite matrices in $\R^{N \times N}$, we define the smallest generalized eigenvalue of $(\bs{A}, \bs{B})$ as 
\begin{align*}
  \lambda_{\mathrm{min}} \left(\bs{A}, \bs{B} \right) \coloneqq \inf_{\bs{x} \in \R^N\backslash \{\bs{0}\}} \frac{\|\bs{x}\|^2_{\bs{A}}}{\|\bs{x}\|_{\bs{B}}^2}.
\end{align*}
Note that under these assumptions on $\bs{A}$ and $\bs{B}$, the above expression is equal to the smallest $\lambda$ such that $\bs{A} \bs{x} = \lambda \bs{B} \bs{x}$ has a nontrivial solution $\bs{x}$, and also we have that
\begin{align*}
  \lambda_{\mathrm{min}}\left(\bs{B}^{-1/2} \bs{A} \bs{B}^{-1/2} \right) = \lambda_{\mathrm{min}} \left(\bs{A}, \bs{B} \right),
\end{align*}
where $\bs{B}^{1/2}$ is the symmetric positive definite matrix square root of $\bs{B}$. Similarly, we use the notation $\lambda_{\mathrm{max}}(\cdot)$ and $\lambda_{\mathrm{max}}(\cdot,\cdot)$ to denote maximum eigenvalues. 

Consider a bounded domain $\Omega \subset \R^d$ with Lipschitz boundary $\partial \Omega$; we are mainly concerned with $d \leq 3$. We have
\begin{align*}
  L^2(\Omega) &\coloneqq \left\{ v: \Omega \rightarrow \R \; \big| \; \| v\|_{L^2} < \infty \right\}, & \| v\|^2_{L^2} &\coloneqq \langle v, v \rangle, & \langle v, w \rangle &\coloneqq \int_\Omega v(x) w(x) \dx{x}.
\end{align*}
We will often write $L^2 = L^2(\Omega)$, and we define $\langle \nabla w, \nabla v \rangle = \sum_{j=1}^d \langle \ppx{x_j} w, \ppx{x_j} v \rangle$, which induces a definition for the $L^2$ norm $\|\nabla v \|$ of vector-valued functions. For brevity will also write $\|v\| = \|v\|_{L^2}$ and $\|\nabla v\| = \|\nabla v\|_{L^2}$. The standard Laplace eigenvalue problem on $\Omega$ with Dirichlet boundary conditions,
\begin{align}\label{eq:eigenproblem}
  \begin{split}
    -\Delta u = \lambda u, \hskip 20pt & \hskip 20pt x \in \Omega, \\
    u(x) = 0, \hskip 20pt & \hskip 20pt x \in \partial \Omega,
  \end{split}
\end{align}
yields an infinite sequence of eigenvalues $0 < \lambda_1 \leq \lambda_2 \cdots $ with associated eigenfunctions $\{\phi_n\}_{n = 1}^\infty$. Here, and in all the following, the differential operator $\Delta$ operates on the $x$ variable. The spectral theorem ensures that the eigenfunctions enjoy $L^2(\Omega)$-orthogonality and completeness, so that 
\begin{align}\label{eq:un-expansion}
  u \in L^2(\Omega) \hskip 10pt \Longrightarrow \hskip 10pt u(x) = \sum_{n=1}^\infty u_n \phi_n(x), \hskip 15pt u_n &= \left\langle u, \phi_n \right\rangle_{L^2(\Omega)},
\end{align}
where we have further assumed that each $\phi_n$ has unit $L^2(\Omega)$ norm.

\subsection{The fractional Laplace problem}
Let $s \in (0,1)$. In this section we describe the spectral definition of the fractional operator $(-\Delta)^s$, on bounded domains, supplemented with homogeneous Dirichlet boundary conditions, see \cite{antil_fractional_2018} for the inhomogeneous case. Related definitions of similar nonlocal or fractional operators can be found in the literature \cite{kwasnicki_ten_2017}. If $u$ has an expansion in eigenfunctions, then a formal definition for application of the fractional operator is,
\begin{align}\label{eq:fraclap-explicit}
  u &= \sum_{n=1}^\infty u_n \phi_n(x) \hskip 10pt \Longrightarrow \hskip 10pt (-\Delta)^s u = \sum_{n=1}^\infty \lambda_n^s u_n \phi_n(x).
\end{align}
We likewise use $\lambda_n$ to define Sobolev spaces of fractional order. With $s \in (0,1)$:
\begin{align*}
  \mathbb{H}^s(\Omega) &\coloneqq \left\{ u \in L^2(\Omega) \;\; \big|\;\; (-\Delta)^{s/2}u \in L^2(\Omega) \right\}, & \mathbb{H}^{-s}(\Omega) &\coloneqq \mathbb{H}^s(\Omega)^\ast,
\end{align*}
where $\mathbb{H}^s(\Omega)^\ast$ denotes the dual space of $\mathbb{H}^s(\Omega)$. With these definitions, $(-\Delta)^s: \mathbb{H}^s(\Omega) \rightarrow \mathbb{H}^{-s}(\Omega)$. For the relation of $\mathbb{H}^s(\Omega)$ to the standard fractional order Sobolev space, see \cite{antil_fractional_2018}.

Given data $f \in L^2(\Omega)$, our main goal is to compute the solution $u$ to
\begin{align}\label{eq:fpde}
  \begin{split}
  (-\Delta)^s u = f,\hskip 20pt & \hskip 20pt x \in \Omega \\
  u = 0, \hskip 20pt & \hskip 20pt x \in \partial \Omega 
  \end{split}
\end{align}
for arbitrary $s \in (0,1)$. Notationally, we omit showing explicit dependence of $u$ on the spatial variable $x$, and only show dependence on the fractional order $s$, which is a parameter. This convention will be used in the remainder of this paper when considering solutions to parameterized PDE's: notational dependence on parameters will be explicit, but that on the spatial variable $x$ will be implicit. Therefore, we let $u(s) \in L^2$ denote the solution $u$ to \eqref{eq:fpde} for a fixed value of $s$.  We will be interested in developing a computational algorithm for computing the \textit{family} or \textit{manifold} of solutions,
\begin{align*}
  U \coloneqq \left\{ u(s) \; \big| \; s \in (0,1) \right\}.
\end{align*}
If $f \in L^2$, then the solution $u$ to \eqref{eq:fpde} has $\mathbb{H}^{2s}$ membership, so that the natural function space in which to study the manifold $U$ is $\cap_{s \in (0,1)} \mathbb{H}^{2 s}$. Hence, all of our investigations will assume $f \in L^2$ and study solutions $u(s)$ as elements of $L^2$. Note that $f \in L^2$ is a stronger requirement than for classical elliptic problems.

In the remainder of this document we will consider the problem \eqref{eq:fpde} with homogeneous Dirichlet boundary conditions. The inhomogeneous boundary case may be handled using the lifting technique in \cite{antil_fractional_2018}, which requires only a small modification of the algorithm proposed here.

\subsection{Kato's integral solution of \eqref{eq:fpde}}
The following remarkable result provides an appealing formula for the solution $u$ to \eqref{eq:fpde}:
\begin{align*}
  u(s) &= \beta(s) \int_{-\infty}^\infty e^{(1-s)y} \left( -\Delta + e^y \right)^{-1} f \dx{y}, & \beta(s) &\coloneqq \frac{\sin \pi s}{\pi}
\end{align*}
which is a reformulation of Kato's formula \cite[Theorem 2 with simplification $\lambda = 0$]{Kato_60}. This representation was first exploited in \cite{bonito_numerical_2015} for designing numerical algorithms, and is derived via a special kind of Dunford-Taylor integral.
To write the above more explicitly, define $q(y)$, for fixed $y \in \R$, as the solution to the \textit{classical} $y$-parameterized PDE,
\begin{subequations}\label{eq:dt-formulation}
  \begin{equation}\label{eq:dt-formulation-a}
  \begin{split}
  -\Delta q(y) + e^y q(y) = f, \hskip 20pt & \hskip 20pt x \in \Omega \\
  q(y) = 0, \hskip 20pt & \hskip 20pt x \in \partial\Omega.
  \end{split}
\end{equation}
Then $u$ is given by 
\begin{align}\label{eq:dt-formulation-b}
  u(s) = \beta(s) \int_{-\infty}^\infty q(y) e^{(1-s)y} \dx{y}.
\end{align}
\end{subequations}
This representation reveals that $u$ is actually just an integral of solutions $q$ to \textit{classical} Laplace-type problems. A solution method employing a discretization of the above formula then only requires solves of classical local PDE's in order to solve the nonlocal problem \eqref{eq:fpde}. The straightforward way to compute the solution via \eqref{eq:dt-formulation-b} is to approximate the integral with a quadrature rule. This would require computing solutions $q(\cdot;y)$ to the PDE \eqref{eq:dt-formulation-a} for many values of the parameter $y$. 

More precisely, let $\{ y_m, \tau_m \}_{m=1}^M$ be a quadrature rule for approximating the integral in \eqref{eq:dt-formulation-b}. We will give precise choices for this quadrature rule soon. Then we can approximate the solution $u(s)$ as 
\begin{align*}
  u(s) \approx u_M(s) \coloneqq \sum_{m=1}^M \tau_m q(y_m) e^{(1-s) y_m}.
\end{align*}
One then needs only to compute the ensemble of functions $\{q(y_m)\}_{m=1}^M$, which are solutions to classical PDE's, in order to approximate the solution to the fractional problem. This is the particular approach adopted in \cite{bonito_numerical_2015}, wherein a sinc quadrature rule is adopted, and associated error bounds are derived.

The observation we make in this paper is that the approach above requires approximately $M$ times the work of a classical problem; when $M$ is large (which can be required when $s$ is small), this may become computationally prohibitive{, see for instance \cite{CWeiss_BvBWaanders_HAntil_2018a}. However, there are by-now standard model reduction approaches that allow one to efficiently compute solutions to parameterized PDE's when the number of queries $M$ is very large; one such method that is directly applicable (to some extent) here is the reduced basis method, which we exploit in Section \ref{sec:rbm}. The next section expresses \eqref{eq:dt-formulation} in a more computationally robust formulation, and proposes a new kind of quadrature for $y$-discretization. We observe in Section \ref{sec:results} that our new quadrature approach is much more efficient than the most efficient strategy considered in \cite{bonito_numerical_2015}.

\section{Fractional Laplace solutions via integral formulation}\label{sec:gq}
Recall that given $s \in (0, 1)$, we seek to evaluate \eqref{eq:dt-formulation-b}, which defines the solution $u(s)$ to the fractional Laplace problem \eqref{eq:fpde}. In this section we describe our algorithm for doing so, which discretizes the $y$ variable using quadrature. The main components of this algorithm come in two stages: first we describe a partitioning formulation for the $y$ variable, followed by a quadrature discretization of the $y$ integral. 

The approach described in this section augments the approach presented in \cite{bonito_numerical_2015}; our improvements include $s$-independent stability in both the continuous and discrete case. For small values of $s$ and large $y$-quadrature rule size, the algorithm in \cite{bonito_numerical_2015} results in discrete operators whose norm becomes very large, which can be problematic for numerical implementation. In our reformulation, the discrete operators are uniformly bounded in $s$ (for any quadrature rule). Second, we replace the sinc quadrature in \cite{bonito_numerical_2015} with a Gaussian quadrature rule.\footnote{The authors in \cite{bonito_numerical_2015} also propose a ``Gaussian" quadrature rule, but theirs is a composite rule, whereas ours is a global rule and is designed differently.}  Our results in section \ref{sec:results} indicate that this quadrature rule is substantially more efficient than sinc quadrature.

\subsection{Partitioning of the $y$ variable}
To make computations numerically stable, we split our parameterized problem \eqref{eq:dt-formulation-a} into regions $y \in (-\infty, 0]$ and $y \in [0, \infty)$. 
To accomplish this introduce a new parameterized PDE for a solution $w$ that is closely related to the solution $q$ from \eqref{eq:dt-formulation-a}
\begin{align}\label{eq:aux-pde}
  \begin{split}
  -\alpha \Delta w(\alpha,\beta) + \beta w(\alpha,\beta) = f, \hskip 20pt & \hskip 20pt x \in \Omega \\
  w(\alpha,\beta) = 0, \hskip 20pt & \hskip 20pt x \in \partial\Omega.
  \end{split}
\end{align}
where $(\alpha,\beta) \in (0, \infty) \times [0, \infty)$ is the parameter. In our computational setting, we will only require $(\alpha,\beta) \in (0, 1]^2$.  Comparing \eqref{eq:aux-pde} with \eqref{eq:dt-formulation-a}, we see that $q(y) = w(1,e^y)$. We now define $w_{\pm}(y)$ as two specializations of $w$ that will be used in the following: 
\begin{subequations}\label{eq:wpm-continuous-def}
\begin{align}
  w_-(y) &\coloneqq w(1,e^{-y}), &  w_+(y) &\coloneqq w(e^{-y}, 1), & y \in [0,\infty). \\
  (-\Delta + e^{-y}) w_-(y) &= f, & (-e^{-y} \Delta + 1) w_+(y) &= f, & 
\end{align}
\end{subequations}
for $x \in \Omega$, with boundary conditions $w_{\pm}(y) = 0$ for $x \in \partial \Omega$. We can now formulate the solution to the fractional PDE \eqref{eq:fpde} in terms of these new quantities.
\begin{lemma}\label{lemma:u-partition}
  The solution $u(s)$ to \eqref{eq:fpde} is given by
  \begin{align}\label{eq:u-partition}
    u(s) &= \sum_{\sigma \in \{-,+\}} \beta_0\left(s_\sigma\right) \int_0^\infty w_\sigma\left(\frac{y}{s_\sigma}\right) W(y) \dx{y} \\\nonumber
         &= \beta_0(s_-) \int_0^\infty w_-\left(\frac{y}{s_-}\right) W(y) \dx{y} + \beta_0(s_+) \int_0^\infty w_+\left(\frac{y}{s_+}\right) W(y) \dx{y},
  \end{align}
  where $\beta_0$, $s_{\pm}$, and $W$ are defined as
  \begin{align*}
    \beta_0(s) &\coloneqq \beta(s)/s = \frac{\sin(\pi s)}{\pi s} = \sinc(s), & s_{\pm} &\coloneqq \frac{1}{2} \pm \left( s - \frac{1}{2} \right) & W(y) &\coloneqq e^{-y}.
  \end{align*}
\end{lemma}
\begin{proof}
  Beginning with \eqref{eq:dt-formulation-b}, we have 
  \begin{align}\nonumber
    u &= \beta(s) \int_{-\infty}^0 e^{(1-s) y} w(1,e^y) \dx{y} + \beta(s) \int_0^\infty e^{(1-s) y} w(1,e^y) \dx{y} \\\nonumber
                  &= (1-s) \beta_0(1-s) \int_0^\infty w_-(y) \exp(-(1-s) y) \dx{y} + s \beta_0(s) \int_0^\infty w_+(y) \exp(-s y)\dx{y} \\\label{eq:integral-division}
                  &= \beta_0(1-s) \int_0^\infty w_-\left(\frac{y}{1-s}\right) W(y) \dx{y} + \beta_0(s) \int_0^\infty w_+\left(\frac{y}{s}\right) W(y) \dx{y},
  \end{align}
  completing the proof.
\end{proof}
Given the domain $\Omega$, we define $C_\Omega$ as the domain's Poincar\'e constant, i.e., the smallest constant such that for every $v \in H_0^1(\Omega)$,
\begin{align}\label{eq:PI}
  \left\| v \right\| &\leq C_\Omega \left\| \nabla v \right\|.
\end{align}
The space $H_0^1(\Omega)$ is the standard Sobolev space of zero-trace $L^2(\Omega)$ functions whose gradients are also in $L^2$. In what follows, we will also need the following quantity,
\begin{align}\label{eq:comega-def}
  \widetilde{C}_\Omega^2 \coloneqq \max\left( 1, C_\Omega^2\right),
\end{align}
which also depends only on $\Omega$. Fixing $(\alpha,\beta)$, the weak formulation of \eqref{eq:aux-pde} seeks a solution $w(\alpha,\beta) = w \in H^1_0\left(\Omega\right)$ as the unique function satisfying the Galerkin formulation,
\begin{align}\label{eq:w-weak-form}
  a\left(w, v; \alpha, \beta\right) &\coloneqq \left\langle f, v \right\rangle, & \forall\; v &\in H^1_0\left(\Omega\right),
\end{align}
with the bilinear form $a(\cdot,\cdot;\alpha, \beta)$ defined as
\begin{align}\label{eq:a-def}
  a\left(w, v; \alpha, \beta\right) \coloneqq \alpha \left\langle \nabla w, \nabla v \right\rangle + \beta \left\langle w, v \right\rangle.
\end{align}
With $\alpha > 0$ and $\beta \geq 0$, the Poincar\`e inequality \eqref{eq:PI} ensures that the coercivity property $a(v,v;y) \geq \alpha \| \nabla v \|^2_{L^2(\Omega)} \geq k \|v\|_{H^1_0(\Omega)}$ holds for some $k > 0$ uniformly in $\beta$, so that standard Lax-Milgram theory then yields a unique $H_0^1(\Omega)$ solution. 

With this setup, we can demonstrate the utility of a formula like \eqref{eq:u-partition} by deriving an $s$-\textit{independent} $L^2$ stability estimate for solutions to \eqref{eq:fpde}.
\begin{proposition}\label{prop:u-stable}
  Assume $f \in L^2$. Then
  \begin{align*}
  \sup_{s \in (0,1)} \| u(s) \| \leq \frac{4 \widetilde{C}_\Omega^2}{\pi} \| f\|.
  \end{align*}
\end{proposition}
\begin{proof}
  The function $w_+(y) \in H_0^1(\Omega)$ is the unique solution to 
  \begin{align*}
    a\left(w_+(y), v; e^{-y}, 1\right) &= \left\langle f, v \right\rangle, & \forall \; v &\in H_0^1(\Omega).
  \end{align*}
  Taking $v = w_+(y)$ and using the Cauchy-Schwarz and Poincar\'e inequalities results in 
  \begin{align*}
    \| f \| \| w_+(y) \| &\geq \left\langle f, w_+(y) \right\rangle = e^{-y} \left\langle \nabla w_+(y), \nabla w_+(y) \right\rangle + \left\langle w_+(y), w_+(y) \right\rangle \\
                                     &\geq \left( 1 + \frac{e^{-y}}{c^2_\Omega}\right) \| w_+(y) \|^2 \geq \| w_+(y)\|^2.
  \end{align*}
  We thus obtain
  \begin{subequations}\label{eq:cpm}
  \begin{align}
    \sup_{y \geq 0} \| w_+(y)\| \leq \| f\|.
  \end{align}
  A similar computation for $w_-$ shows that
  \begin{align}
    \sup_{y \geq 0} \| w_-(y)\| \leq C^2_\Omega \| f\|.
  \end{align}
  \end{subequations}
  Therefore, taking the $L^2$ norm in \eqref{eq:u-partition} and using the triangle inequality yields
  \begin{align*}
    \| u(s) \| &\leq \beta_0(s_-) \int_0^\infty \left\| w_-\left(\frac{y}{s_-}\right) \right\| W(y) \dx{y} + \beta_0(s_+) \int_0^\infty \left\| w_+\left(\frac{y}{s_+}\right)\right\| W(y) \dx{y} \\ 
                     &\stackrel{\eqref{eq:cpm}}{\leq} \beta_0(s_-) C^2_\Omega \|f\| \int_0^\infty W(y) \dx{y} + \beta_0(s_+) \|f\| \int_0^\infty W(y) \dx{y} \\
    &\leq \max( 1 , C^2_\Omega ) \| f\| \left[ \beta_0(s_-) + \beta_0(s_+)\right],
  \end{align*}
  where the third inequality uses the fact that $W$ is a probability density on $[0, \infty)$. From the above, we immediately obtain the desired result by noting that 
  \begin{align*}
    \beta_0(s_-) + \beta_0(s_+) &= \frac{\sin(\pi s)}{\pi s(1-s) } \leq \frac{4}{\pi}, & s \in (0,1).
  \end{align*}
  The proof is complete.
\end{proof}

\subsection{Spatial discretization}\label{ssec:space-discretization}
In this section we employ a spatial discretization to the result of Lemma \ref{lemma:u-partition}.  
We proceed to discretize \eqref{eq:a-def} using a finite element method. 
Let $T_\Omega$ be a conforming triangulation of $\Omega$ with $K$ elements. 
We assume each element in the triangulation is isoparametrically equivalent to a standard canonical triangle/tetrahedron. 
For a fixed polynomial degree $k \geq 1$, we define the finite element space
\begin{align*}
  V = \left\{ v \in C\left(\overline{\Omega}\right) \;\; \big| \;\; v|_e \in P_k\left(e\right)\;\; \forall e \in T_\Omega, \ v|_{\partial\Omega} = 0 \right\},
\end{align*}
where $P_k(e)$ is the space of polynomials of degree $k$ or less over the element $e \in T_\Omega$. 
Let $\calN = \dim V$. 
The finite element-discretized version of \eqref{eq:w-weak-form} is the Galerkin formulation seeking $w^{\calN} \in V$ satisfying
\begin{align}\label{eq:w-galerkin}
  a\left(w^{\calN}, v; \alpha, \beta \right) &= \left\langle f, v \right\rangle, & \forall v &\in V.
\end{align}
Let $w^\calN(\alpha,\beta)$ be expressed as a linear expansion,
\begin{align}\label{eq:truth-expansion}
  w^\calN(\alpha,\beta) &= \sum_{n=1}^{\calN} w^\calN_n(\alpha,\beta) \psi_n,
\end{align}
where $\{\psi_n\}_{n=1}^\calN$ is a basis for $V$, e.g., a basis comprised of compactly supported piecewise polynomials. Collecting the linear degrees of freedom of $w^\calN(y) \in V$ in the $\calN$-dimensional vector $\bs{w}^\calN$, then this vector satisfies the linear system,
\begin{subequations}\label{eq:truth-discretization}
\begin{align}
  \bs{A}(y) \bs{w}^{\calN}(\alpha,\beta) &= \bs{f}, & (\bs{f})_j &= \left\langle f, \psi_j \right\rangle,
\end{align}
where the $\calN \times \calN$ matrix $\bs{A}(y)$ has entries,
\begin{align}
  (\bs{A}(y))_{j,k} &= a\left( \psi_k, \psi_j \right) = \alpha \left\langle \nabla \psi_k, \nabla \psi_j \right\rangle + \beta \left\langle \psi_k, \psi_j \right\rangle \coloneqq \alpha (\bs{S})_{j,k} + \beta (\bs{M})_{j,k},
\end{align}
\end{subequations}
for $j, k = 1, \ldots, \calN$. Above we have defined the $\calN \times \calN$ $y$-independent stiffness and mass matrices $\bs{S}$ and $\bs{M}$, respectively. Both $\bs{S}$ and $\bs{M}$ are symmetric and positive-definite. 

The matrices $\bs{S}$ and $\bs{M}$ can be used to define a ``discretized" Poincar\'e constant $C_\calN$ by using a standard Rayleigh quotient argument:
\begin{align}\label{eq:CN-def}
  \frac{1}{C_\Omega^2} = \inf_{v \in H_0^1\backslash\{0\}} \frac{\|\nabla v\|^2}{\|v\|^2} \leq \inf_{v \in V} \frac{\|\nabla v\|^2}{\|v\|^2} = \inf_{\bs{v} \in \R^\calN\backslash\{\bs{0}\}} \frac{\bs{v}^T \bs{S} \bs{v}}{\bs{v}^T \bs{M} \bs{v}} = \lambda_{\mathrm{min}}\left(\bs{S}, \bs{M}\right) \eqqcolon \frac{1}{C_\calN^2},
\end{align}
hence leading to the inequalities,
\begin{align}\label{eq:CNtilde-def}
  C_{\calN} &\leq C_\Omega, & \max\left(1, C_{\calN}^2\right) \stackrel{\eqref{eq:comega-def}}{\leq} \widetilde{C}_\Omega^2.
\end{align}
Our estimates for $x$-discrete quantities will involve $C_{\calN}$, but we will sometimes use the above inequality to bound quantities in terms of $C_\Omega$. Bounds involving $C_\Omega$ emphasize independence of the $x$-discretization. Bounds involving $C_\calN$ emphasize the explicit computability of the bounds, since $C_\calN$ is equal to an extremal eigenvalue of finite element matrices, which is computable with iterative eigenvalue solvers. 

The maximum generalized eigenvalue of $(\bs{S}, \bs{M})$ will also play a small role in our estimates. In analogy with \eqref{eq:CN-def} we define
\begin{align}\label{eq:KN-def}
  \frac{1}{K_\calN^2} \coloneqq \lambda_{\mathrm{max}}\left(\bs{S}, \bs{M}\right).
\end{align}
Note that $K_\calN$ in general tends to 0 as $\calN \uparrow \infty$. 

From the discretization of $w$ we derive discretizations of $w_{\pm}(y)$ defined in \eqref{eq:wpm-continuous-def}. 
Thus, finite element discretizations for $w_{\pm}$ are specializations of those for $w$. In particular, we define
\begin{align*}
  w^\calN_-(y) &\coloneqq w^\calN\left(1, e^{-y}\right) \in V, & w^\calN_+(y) &\coloneqq w^\calN\left(e^{-y}, 1\right) \in V.
\end{align*}
Denote by $\bs{w}_{\pm}^\calN(y)$ the $\calN$-dimensional vectors that are solutions to the linear systems,
\begin{align}\label{eq:wpm-discrete-def}
  \left( \bs{S} + e^{-y} \bs{M}\right) \bs{w}^\calN_-(y) &= \bs{f}, & \left( e^{-y} \bs{S} + \bs{M}\right) \bs{w}^\calN_+(y) &= \bs{f}, & y &\in [0, \infty),
\end{align}
so that, akin to \eqref{eq:truth-expansion}, we have 
\begin{align*}
  w^\calN_{\pm}(y) &= \sum_{j=1}^\calN w^\calN_{j,\pm}(y) \psi_j, & \bs{w}_{\pm}^\calN(y) &= \left( w_{1,\pm}^\calN(y), \ldots, w_{\calN,\pm}^\calN(y) \right)^T.
\end{align*}
We can now codify the fact that the solutions $w^\calN_{\pm}(y)$ are $L^2$-stable \textit{uniformly} in $y$.
\begin{lemma}
  Assume $f \in L^2(\Omega)$. Then
  \begin{subequations}\label{eq:wl2}
    \begin{align}
      \label{eq:wl2-a}
      \sup_{y \geq 0} \| w^\calN_{+}(y) \| &\leq \| f \|, \\\label{eq:wl2-b}
      \sup_{y \geq 0} \| w^\calN_{-}(y) \| &\leq C_\calN^2 \| f \|. 
    \end{align}
  \end{subequations}
\end{lemma}
\begin{proof}
  The result can be obtained by considering the discrete form \eqref{eq:wpm-discrete-def}. To begin we relate the $L^2$ norm of $f$ to the Euclidean $\ell^2$ norm of $\bs{f}$. Let $P_V$ denote the $L^2$-orthogonal projector onto $V$. Then:
  \begin{align*}
    f_j = \left\langle f, \psi_j \right\rangle \hskip 5pt &\Longrightarrow \hskip 5pt \| P_V f\|^2 = \bs{f}^T \bs{M}^{-1} \bs{f}. \\
    w_+^\calN(y) = \sum_{j=1}^\calN w_{j,+}^\calN(y) \psi_j \hskip 5pt &\Longrightarrow \hskip 5pt \| w_+^\calN(y)\|^2 = \left( \bs{w}^\calN_+\right)^T \bs{M} \left( \bs{w}^\calN_+\right).
  \end{align*}
  Thus we have 
  \begin{align}\label{eq:lemma-dc-norms}
    \begin{split}
    \| \bs{f} \|^2_{\bs{M}^{-1}} = \| P_V f \|^2 \leq \| f\|^2, \\
    \| \bs{w}^\calN_+ \|^2_{\bs{M}} = \| w^\calN_+ \|^2,
    \end{split}
  \end{align}
  Now since $\bs{M}$ is symmetric and positive-definite, it has a unique symmetric positive-definite square root $\bs{M}^{1/2}$. Thus:
  \begin{align*}
    \left( e^{-y} \bs{S} + \bs{M} \right) \bs{w}^\calN_+(y) = \bs{f} \hskip 5pt \Longrightarrow \left( e^{-y} \bs{M}^{-1/2} \bs{S} \bs{M}^{-1/2} + \bs{I} \right) \left( \bs{M}^{1/2} \bs{w}^\calN_+(y) \right) = \bs{M}^{-1/2} \bs{f}.
  \end{align*}
  This in turn implies:
  \begin{align*}
    \| \bs{w}^\calN_+ \|_{\bs{M}} \leq \frac{1}{\lambda_{\mathrm{min}}\left( e^{-y} \bs{M}^{-1/2} \bs{S} \bs{M}^{-1/2} + \bs{I} \right)} \| \bs{f}\|_{\bs{M}^{-1}}.
  \end{align*}
  Since $\bs{M}^{-1/2} \bs{S} \bs{M}^{-1/2}$ is symmetric and positive-definite, we have
  \begin{align*}
    \lambda_{\mathrm{min}}\left( e^{-y} \bs{M}^{-1/2} \bs{S} \bs{M}^{-1/2} + \bs{I} \right) \geq \lambda_{\mathrm{min}}\left( \bs{I} \right) = 1.
  \end{align*}
  Therefore, 
  \begin{align*}
    \| \bs{w}^\calN_+ \|_{\bs{M}} \leq \frac{\| \bs{f}\|_{\bs{M}^{-1}}}{\lambda_{\mathrm{min}}\left( e^{-y} \bs{M}^{-1/2} \bs{S} \bs{M}^{-1/2} + \bs{I} \right)} \leq \| \bs{f} \|_{\bs{M}^{-1}},
  \end{align*}
  which, when combined with \eqref{eq:lemma-dc-norms} yields \eqref{eq:wl2-a}. A similar computation for $w_-(y)$ yields \eqref{eq:wl2-b} by using the definition of $C_\calN$ in \eqref{eq:CN-def}.
\end{proof}
The result above gives the stability of an algorithm that uses $w^{\calN}_{\pm}(y)$ as a spatial discretization. In particular, consider the following semi-discrete approximation of $u(s)$,
\begin{align}\label{eq:u-semidiscrete}
  \bs{\widetilde{u}}^\calN(s) &\coloneqq \sum_{\sigma \in \{+,-\}} \beta_0(s_\sigma) \int_0^\infty \bs{w}^\calN_\sigma\left(\frac{y}{s_\sigma}\right) W(y) \dx{y}, & 
  \widetilde{u}^\calN(s) &\coloneqq \sum_{j=1}^\calN \widetilde{u}^\calN_j(s) \psi_j \in V.
\end{align}
A fully discrete scheme, introduced in the next section, would discretize the $y$ variable. The following result mirrors the stability estimate of Proposition \ref{prop:u-stable}, showing $s$-uniform $L^2$ stability of the semi-discrete solution.
\begin{proposition}
  Assume $f \in L^2$. Then
  \begin{align*}
    \sup_{s \in (0,1)} \| \widetilde{u}^\calN(s) \| \leq \frac{4 \widetilde{C}_\Omega^2}{\pi} \| f\|.
  \end{align*}
\end{proposition}
The proof, which we omit, is almost exactly the same as for Proposition \ref{prop:u-stable} using the discrete stability estimates \eqref{eq:wl2}; the only novelty is the need to exercise the inequality \eqref{eq:CN-def}.

Having described the spatial discretization, we now proceed to describe a discretization for the $y$-integrals in \eqref{eq:u-partition}.

\subsection{Quadrature for the $y$-integral}
We will use an $M$-point $W$-Gaussian quadrature rule to discretize the integrals in \eqref{eq:u-partition}. The weight function $W(y)$ is a weight function associated to a classical family of orthogonal polynomials: Laguerre polynomials. Let $\{p_n\}_{n \geq 0}$ denote the family of Laguerre polynomials, orthonormal under the weight $W$, i.e., 
\begin{align*}
  \int_0^\infty p_n(y) p_m(y) W(y) \dx{y} &= \delta_{m,n}, & m, n &\in \N_0,
\end{align*}
where $\delta_{m,n}$ is the Kronecker delta function. Like all orthogonal polynomials, the family $\{p_n\}_{n \geq 0}$ satisfies a three-term recurrence formula,
\begin{align*}
  y p_n(y) &= b_{n+1} p_{n+1}(y) + a_{n+1} p_n(y) + b_n p_{n-1}(y), & n&\geq 0,
\end{align*}
with $p_0 \equiv 1$ and $p_{-1} \equiv 0$. The recurrence coefficients $(a_n,b_n)$ are explicitly known:
\begin{align*}
  b_0 &= 1, & b_n &= n & a_n &= 2n - 1, & n &\geq 1.
\end{align*}
Among the properties of orthogonal polynomial families is the existence of a unique $M$-point quadrature rule with optimal polynomial exactness, the \textit{Gaussian} quadrature rule. This rule has abscissae and weights, $(y_m, \tau_m)_{m=1}^M$, respectively, and integrates polynomials up to degree $2M -1$ exactly: 
\begin{align*}
  \int_0^\infty p(y) W(y) \dx{y} &= \sum_{j=1}^M \tau_j p(y_j), & p &\in \mathrm{span}\{1, y, \ldots, y^{2 M-1} \}.
\end{align*}
Although $y_j$ and $\tau_j$ depend on the value of $M$, we omit this explicit notational dependence. This rule can also be easily computed with knowledge of the recurrence coefficients. Defining the $M \times M$ symmetric tridiagonal Jacobi matrix,
\begin{align*}
  \bs{J}_M \coloneqq \left( \begin{array}{ccccc} a_1 & b_1 & 0 & \cdots & \\
  b_1 & a_2 & b_2 & \ddots & \\
      & \ddots & \ddots & \ddots & \\
      & 0 & \cdots & b_{M-1} & a_M \end{array}\right),
\end{align*}
consider its associated eigenvalue decomposition,
\begin{subequations}\label{eq:gq-nodes-weights}
\begin{align}
\bs{J}_M &= \bs{V} \bs{\Lambda} \bs{V}^T, & \bs{\Lambda} &= \mathrm{diag}(\lambda_1, \ldots, \lambda_M), & \bs{V} = \left[\begin{array}{cccc} \bs{v}_1 & \bs{v}_2 & \cdots & \bs{v}_M \end{array}\right],
\end{align}
where $\bs{V}$ is unitary since $\bs{J}_M$ is symmetric. The Gaussian quadrature rule can be computed from these quantities. In particular,
\begin{align}
  y_j &= \lambda_j, & \tau_j = v_{j,1}^2 b_0^2,
\end{align}
\end{subequations}
where $v_{j,1}^2$ is the first component of the vector $\bs{v}_j$. To compute an $M$-point quadrature rule for $W$ thus requires a size-$M$ eigenvalue computation. Since $W$ is a probability density function, then likewise $\sum_{j=1}^M \tau_j = 1$, and furthermore each $\tau_j$ is non-negative by \eqref{eq:gq-nodes-weights}.

Now let $M_{-}$ and $M_+$ be the number of quadrature points used to approximate the ``$-$" and ``$+$" integrals in \eqref{eq:u-partition}, respectively. This results in two sets of $W$-Gaussian quadrature rules,
\begin{align*}
  \left( y_{j,-}, \tau_{j,-} \right)_{j=1}^{M_-}, \hskip 15pt \left( y_{j,+}, \tau_{j,+} \right)_{j=1}^{M_+}.
\end{align*}
We apply these rules to the integrals in \eqref{eq:u-semidiscrete}, resulting in the fully discrete approximation,
\begin{align}\label{eq:gq-truth-solution}
  \bs{u}^\calN(s) &\coloneqq \sum_{\sigma \in \{+,-\}} \beta_0(s_\sigma) \sum_{j=1}^{M_\sigma} w_\sigma\left(\frac{y_{j,\sigma}}{s_\sigma}\right) W(y), & u^\calN(s) &\coloneqq \sum_{j=1}^\calN u^\calN_j(s) \psi_j \in V 
\end{align}
We emphasize that the $y$-discretization does not suffer from any numerical instabilities as $s \uparrow 1$ or $s \downarrow 0$: the weights $\tau_{j,\pm}$ are positive, no larger than 1, and independent of $s$, $\beta_0(1-s)$ and $\beta_0(s)$ are just sinc functions, and $w_{\pm}(y)$ has bounded $L^2$ norm for all $y \geq 0$, i.e., for all inputs. The following codifies this stability.
\begin{proposition}\label{prop:u-stability}
  Assume $f \in L^2(\Omega)$. Then 
  \begin{align}\label{eq:uN-stability}
    \sup_{s \in (0,1)} \| u^\calN(s) \| &\leq \frac{4 \widetilde{C}_\Omega^2}{\pi} \| f\|,
  \end{align}
\end{proposition}
\begin{proof}
  The proof is very similar to the proof for proposition \eqref{prop:u-stable}. Take the $\|\cdot\|_{\bs{M}}$ norm on both sides of \eqref{eq:gq-truth-solution}, use the triangle inequality and \eqref{eq:wl2}, and note that the quadrature weights $\tau_{j,\pm}$ all satisfy $0 \leq \tau_{j,\pm} \leq 1$ since they are all positive and $W$ is a probability density. Note that \eqref{eq:uN-stability} holds for any quadrature rule for the $y$ variable if augmented by a multiplicative constant equal to the quadrature condition number (sum of absolute value of weights).
\end{proof}
Just as in \cite{bonito_numerical_2015}, assuming we have the ability to compute $\bs{w}^\calN_{\pm}$, then the formulation above immediately yields an algorithm to compute $u^\calN(s)$ via \eqref{eq:gq-truth-solution}.
\begin{remark}\label{rem:general-a}
  All of our results extend to the case when we solve \eqref{eq:fpde} but replacing $-\Delta$ with a general elliptic operator $\mathcal{E}$ that (i) satisfies the coercivity condition $\left\langle \mathcal{E} v, v \right\rangle \geq \alpha \| v \|_{H_0^1}$ for $\alpha > 0$ and (ii) can be associated with a symmetric variational bilinear form. In this more general case, we need only replace all the instances of $\widetilde{C}^2_\Omega$ with $\widetilde{C}^2_\Omega / \alpha$. The matrix $\bs{S}$ should likewise be replaced with the associated matrix defined from the bilinear form of $\mathcal{E}$.
\end{remark}



\subsection{Algorithm summary}\label{ssec:gq:algorithm}
The sections above identify an algorithm for computing $u^\calN(s)$ in \eqref{eq:gq-truth-solution}. We summarize the procedure in Algorithm \ref{alg:GQ}. The discrete solution adheres to the stability bound in Proposition \ref{prop:u-stability}. Note that we have not yet described how one should decide on values for $M_{\pm}$. We provide a concrete computational strategy for accomplishing this in the next section. However, one of our goals in our numerical results section is to compare our algorithm to existing ones, which determine $M_\pm$ using \eqref{eq:N-quad}. 
\begin{algorithm}
  \caption{GQ algorithm: Produces solution to the fractional Laplace problem \eqref{eq:fpde}.
    \label{alg:GQ}}
  \begin{algorithmic}[1]
    \Require{Availability of a discrete solution $\bs{w}^{\calN}(\alpha,\beta)$ from the formulation \eqref{eq:w-galerkin}.}
    \Statex
    \Function{FracLapGQ}{$s$}
      \State Determine $M_{\pm}$, e.g., via \eqref{eq:N-quad}.
      \State Generate quadrature rules $\left( y_{j,\pm}, \tau_{j,\pm} \right)_{j=1}^{M_\pm}$ using \eqref{eq:gq-nodes-weights}.
      \For{$j \gets 1 \textrm{ to } M_-$}
      \State Compute $\bs{w}^{\calN}_-\left(\frac{y_{j,-}}{1-s}\right) = \bs{w}^\calN\left(1, \exp\left(-\frac{y_{j,-}}{1-s}\right)\right)$ from \eqref{eq:wpm-discrete-def} or \eqref{eq:truth-discretization}.
      \EndFor
      \For{$j \gets 1 \textrm{ to } M_+$}
      \State Compute $\bs{w}^{\calN}_+\left(\frac{y_{j,+}}{s}\right) = \bs{w}^\calN\left(\exp\left(-\frac{y_{j,+}}{s}\right),1\right)$ from \eqref{eq:wpm-discrete-def} or \eqref{eq:truth-discretization}.
      \EndFor
      \State Compute $\bs{u}^\calN(s)$ from \eqref{eq:gq-truth-solution}.
    \State \Return{$\bs{u}^\calN(s)$}
    \EndFunction
  \end{algorithmic}
\end{algorithm}

\section{Error due to quadrature discretization}\label{sec:quad}
The formulation \eqref{eq:gq-truth-solution} is our numerical approximation to compute solutions to \eqref{eq:fpde}. This formulation is a discretization over both the $x$ and $y$ variables (via a finite element formulation and a quadrature rule, respectively). To understand the error that the $y$ discretization contributes, we analyze the discrepancy between $\widetilde{u}^\calN(s)$ and $u^\calN(s)$.

To proceed, we need an auxiliary function that measures the absolute error between a size-$M$ quadrature rule $\left(y_j, \tau_j\right)_{j=1}^M$ and the exact integral applied to a particular function:
\begin{align}\label{eq:g-def}
  g_M(a, b) &\coloneqq \left| \int_0^\infty \frac{W(y)}{1 + a e^{-b y}} \dx{y} - \sum_{j=1}^{M} \frac{\tau_j}{1 + a e^{-b y_j}} \right|, & (a,b) \in (0, \infty) \times (1, \infty).
\end{align}
We also need to define intervals on the real line enclosing the spectrum of some discretized operators. Recalling the definitions of $C_\calN$ and $K_\calN$ in \eqref{eq:CN-def} and \eqref{eq:KN-def}, respectively, define intervals $I$ and $I_{\pm}$ as
\begin{align*}
  I_- &= \left[ K^2_\calN, C^2_\calN \right] \subset (0, \infty), &
  I_+ &= \left[ \frac{1}{C_\calN^2}, \frac{1}{K_\calN^2} \right] \subset (0, \infty), &
  I &= I_- \bigcup I_+
\end{align*}
The error committed by the quadrature rule can be understood by studying the quantity,
\begin{align}\label{eq:G-def}
  G_{\pm}(M,s) \coloneqq \beta_0\left(s_{\pm}\right) \sup_{a \in I_{\pm}} g_M\left(a, \frac{1}{s_\pm}\right).
\end{align}
The precise statement is as follows.
\begin{proposition}\label{prop:qerror}
  Assume $f \in L^2$.  Then for each $s \in (0,1)$,
  \begin{align}\label{eq:quad-error}
    \left\| u^\calN(s) - \widetilde{u}^\calN(s) \right\| \leq \widetilde{C}_\Omega^2 \|f\| \sum_{\sigma \in \{+,-\}} G_\sigma\left(M_\sigma,s\right) 
  \end{align}
  and therefore,
  \begin{align}\label{eq:sup-quad-error}
    \sup_{s \in (0,1)} \left\| u^\calN(s) - \widetilde{u}^\calN(s) \right\| \leq \frac{4\widetilde{C}_\Omega^2 \|f\|}{\pi} \max_{\sigma \in \{+,-\}} \sup_{a \in I, b \in (1, \infty)} g_{M_\sigma}(a,b).
  \end{align}
\end{proposition}
\begin{proof}
  The same argument that produces the relations \eqref{eq:lemma-dc-norms} implies that 
  \begin{align*}
    \left\| u^\calN(s) - \widetilde{u}^\calN(s) \right\|_{L^2} = \left\| \bs{u}^\calN(s) - \bs{\widetilde{u}}^\calN(s) \right\|_{\bs{M}},
  \end{align*}
  so we proceed to study the quantity on the right-hand side. The difference between the ``$-$" integral contributions in $\bs{u}^\calN(s) - \bs{\widetilde{u}}^\calN(s)$ is proportional to 
  \begin{align*}
    \int_0^\infty \bs{w}^\calN_-\left(\frac{y}{s_-}\right) W(y) \dx{y} - \sum_{j=1}^{M_-} \bs{w}^\calN_-\left( \frac{y_{j,-}}{s_-} \right) \tau_{j,-}.
  \end{align*}
  We can express the solution $\bs{w}_-^\calN(y)$ as 
  \begin{align*}
    \bs{w}_-^\calN(y) &= \left[ \bs{S} + e^{-y/s_-} \bs{M} \right]^{-1} \bs{f} = \bs{M}^{-1/2} \left[ \bs{A} + e^{-y/s_-} \bs{I} \right]^{-1} \bs{M}^{-1/2},
  \end{align*}
  where we have defined $\bs{A} \coloneqq \bs{M}^{-1/2} \bs{S} \bs{M}^{-1/2}$. The matrix $\bs{A}$ is symmetric and positive definite, and thus has an eigenvalue decomposition
  \begin{align*}
    \bs{A} &= \bs{W} \bs{T} \bs{W}^T, & \bs{W} \bs{W}^T = \bs{I}.
  \end{align*}
  Then further manipulation of the $\bs{w}_-^\calN(y)$ expression yields
  \begin{align*}
    \bs{w}_-^\calN(y) &= \bs{M}^{-1/2} \bs{W} \bs{H}\left(\frac{y}{s_-}\right) \bs{W}^T \bs{M}^{-1/2},
  \end{align*}
  where $\bs{H}(y)$ is a diagonal matrix having entries
  \begin{align*}
    \left(\bs{H}(y) \right)_{j,j} = \frac{1}{t_j + e^{-y}} = \frac{1}{t_j} \left[ 1 + \frac{1}{t_j} e^{-y} \right]^{-1}.
  \end{align*}
  Thus, 
  \begin{align*}
    \bigg\| \int_0^\infty &\bs{w}^\calN_-\left(\frac{y}{s_-}\right) \dx{y} - \sum_{j=1}^{M_-} \bs{w}^\calN_-\left( \frac{y_{j,-}}{s_-} \right) \tau_{j,-} \bigg\|_{\bs{M}} \\
                          &\leq \left\| \bs{W} \left[ \int_0^\infty \bs{H}(y/s_-) W(y) \dx{y}  - \sum_{j=1}^{M_-} \tau_{j,-} \bs{H}\left(\frac{y_{j,-}}{s_-}\right)\right] \bs{W}^T \right\|  \left\| \bs{M}^{-1/2} \bs{f} \right\| \\ 
    &= \left\| \left[ \int_0^\infty \bs{H}(y/s_-) W(y) \dx{y}  - \sum_{j=1}^{M_-} \tau_{j,-} \bs{H}\left(\frac{y_{j,-}}{s_-}\right)\right] \right\| \left\| \bs{f} \right\|_{\bs{M}^{-1}} \\
    &\stackrel{\eqref{eq:lemma-dc-norms}}{\leq} \left\| \left[ \int_0^\infty \bs{H}(y/s_-) W(y) \dx{y}  - \sum_{j=1}^{M_-} \tau_{j,-} \bs{H}\left(\frac{y_{j,-}}{s_-}\right)\right] \right\| \|f\|_{L^2} \\
    &= \left\| f\right\|_{L^2} \max_{j=1,\ldots,\N} \frac{1}{t_j} g_-\left(\frac{1}{t_j}, \frac{1}{s_-} \right),
  \end{align*}
  The first inequality is sub-multiplicativity of the $\|\cdot\|$ matrix norm, and the first equality uses the invariance of the same norm under unitary transformations. A similar computation for the ``$+$" quantities yields
  \begin{align*}
    \left\| \int_0^\infty \bs{w}^\calN_+\left(\frac{y}{s_+}\right) \dx{y} - \sum_{j=1}^{M_+} \bs{w}^\calN_+\left( \frac{y_{j,+}}{s_+} \right) \tau_{j,+} \right\|_{\bs{M}} &\leq \left\| f\right\|_{L^2} \max_{j \in [\calN]} g_+\left(t_j, \frac{1}{s_+} \right) .
  \end{align*}
  The combination of these results implies
  \begin{align*}
    \left\| u^\calN(s) - \widetilde{u}^\calN(s) \right\|_{L^2} &\leq \beta_0(s_-) \left\| \int_0^\infty \bs{w}^\calN_-\left(\frac{y}{s_-}\right) \dx{y} - \sum_{j=1}^{M_-} \bs{w}^\calN_-\left( \frac{y_{j,-}}{s_-} \right) \tau_{j,-} \right\|_{\bs{M}} \|f\|_{L^2} \\ 
                                                               &+ \beta_0(s_+) \left\| \int_0^\infty \bs{w}^\calN_-\left(\frac{y}{s_+}\right) \dx{y} - \sum_{j=1}^{M_+} \bs{w}^\calN_+\left( \frac{y_{j,+}}{s_+} \right) \tau_{j,+} \right\|_{\bs{M}} \|f\|_{L^2} \\
                                                               &\leq \frac{\|f\|_{L^2}}{\lambda_{\mathrm{min}}\left(\bs{S}, \bs{M}\right)} \sup_{a \in I_-} g_-\left( a, \frac{1}{s_-} \right) + \|f\|_{L^2} \sup_{a \in I_+} g_+\left( a, \frac{1}{s_+}\right).
  \end{align*}
  Using the inequality \eqref{eq:CNtilde-def} yields the result.
\end{proof}
The summation on the right-hand side of \eqref{eq:quad-error} can be computed independent of the data $f$, and requires only knowledge of the extremal eigenvalues of the discrete operator, cf. \eqref{eq:CN-def} and \eqref{eq:KN-def}. While we cannot at present provide a theoretical estimate of this error, we numerically investigate the behavior of this error on $M_{\pm}$ in our numerical results section.
\begin{remark}
  Comparing \eqref{eq:quad-error} with the stability bound \eqref{eq:uN-stability} suggests that many of the factors in \eqref{eq:quad-error} appear due to bounding the error relative to $\| \widetilde{u}^\calN \|_{L^2}$. Thus the supremum over $g$ is the factor that arises due to the quadrature error.
\end{remark}
\begin{remark}
  The result \eqref{eq:quad-error} also shows that the error between $u^\calN(s)$ and $\widetilde{u}^\calN(s)$ is stable independent of $s$ since
  \begin{align*}
    g_M(a,b) \leq \left| \int_0^\infty W(y) \right| + \left| \sum_{j=1}^{M} \tau_{j} \right| = 2,
  \end{align*}
  uniformly in $a$, $b$, and $M$. Thus, 
  \begin{align*}
    \sup_{s \in (0,1)} \left\| u^\calN(s) - \widetilde{u}^\calN(s) \right\|_{L^2} \leq \frac{8 \widetilde{C}_\Omega^2}{\pi} \| f\|_{L^2}.
  \end{align*}
  This again suggests that, independent of all discretization parameters, our numerical algorithm is stable. However, a rigorous convergence analysis for our quadrature rule is not yet available.
\end{remark}

\subsection{Empirical behavior of quadrature error}
The main result from Proposition \ref{prop:qerror} is that the error in the fully discrete approximation \eqref{eq:gq-truth-solution} that is due to the $y$-quadrature discretization is computable without solving any PDE's, assuming that the extremal generalized eigenvalues of $(\bs{S}, \bs{M})$, coded in the quantities $C_\calN$ and $K_\calN$, are known. In particular, this implies that most details of the spatial discretization need not be utilized to understand the quadrature error; we only require extremal eigenvalues of discretized operators. 

We empirically investigate the accuracy of the quadrature rule in this section. Throughout our tests, we will use the following values:
\begin{align*}
  C^2_{\calN} = 2 &\Longleftrightarrow \lambda_{\mathrm{min}}\left(\bs{S}, \bs{M}\right) = \frac{1}{2}, & 
  K^2_{\calN} = \frac{1}{10^6} &\Longleftrightarrow \lambda_{\mathrm{max}}\left(\bs{S}, \bs{M}\right) = 10^6.
\end{align*}
Since $C_\calN$ is bounded above by the analytical Poincar\'e constant of the domain $C_\Omega$, then choosing this $\mathcal{O}(1)$ quantity for $C_\calN$ is reasonable. The choices above make the intervals $I_{\pm}$ defined in Proposition \ref{prop:qerror} explicit. The finite element discretization from our numerical experiments in Section \ref{sec:results} results in values $C^2_\calN = 0.0506$ and $K^2_\calN = 2.36 \times 10^{-6}$.

We note that $G_{\pm}$ in \eqref{eq:G-def} can be numerically approximated for each $(M,s)$ by replacing the supremum over $a$ with the maximum over a discrete mesh. 
We compute the supremums in $G_\pm$ by discretizing the intervals $I_\pm$ with $200$ logarithmically spaced points. (I.e., $\log I_{\pm}$ is replaced with 200 equispaced points.) This discretization then allows us to compute $G_{\pm}$, and hence allows us to compute approximations to the bound in \eqref{eq:quad-error}. In Figure \ref{fig:Gvals}, we show the behavior of $G_{\pm}$ as a function of $(M,s)$. We note that ensuring small values of $G_+$ requires more quadrature points when $s$ is close to 0. In contrast, controlling $G_-$ requires more quadrature points when $s$ is close to 1. However, the behavior of $G_+$ for small $s_+ = s$ is more restrictive than the behavior of $G_-$ for small $s_- = 1-s$. Thus, we expect that $G_+$ is the term that requires more computational investment to guarantee a certain error level.

\begin{figure}[htbp]
  \begin{center}
    \resizebox{0.85\textwidth}{!}{
      \includegraphics{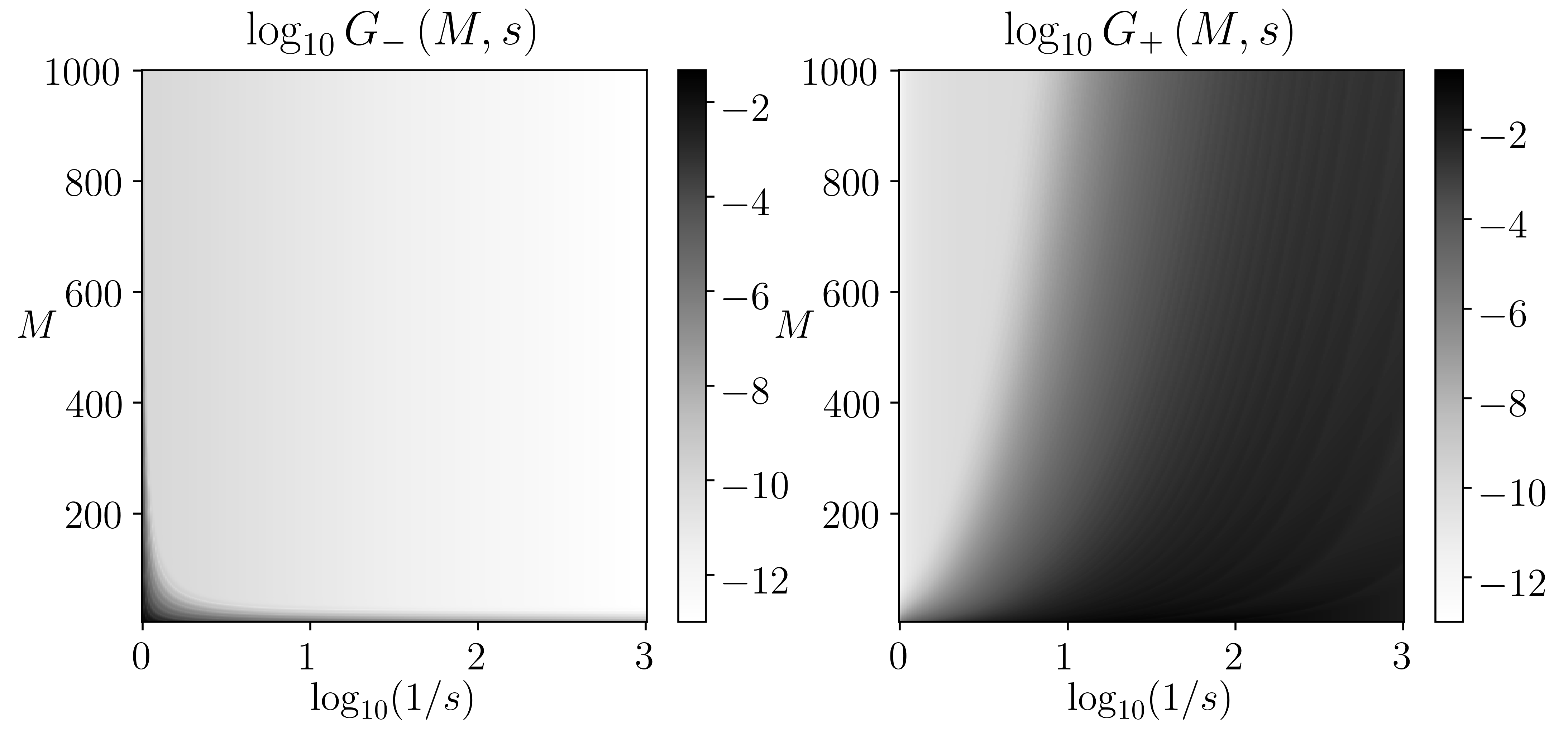}
    }
  \end{center}
  \caption{\label{fig:Gvals} Values of $\log_{10} G_{\pm}$ defined in \eqref{eq:G-def} as a function of $(M,s)$. We show $s$-dependence as $\log_{10}(1/s)$ since the error behavior for small $s$ is the most restrictive. We observe that, for fixed $M$, $G_{\pm}$ has large values when $s_\pm$ is small.}
\end{figure}

To explore this further, we define the smallest value\footnote{To avoid computational effects of oscillating errors due to, e.g., even/odd parity of the quadrature rule, we actually compute the smallest value of $M$ so that $M$, $M+1$, $M+2$, and $(M+3)$-point quadrature rules all achieve the stated accuracy requirement.} of $M$ needed to assure a given error level $\delta$:
\begin{align}\label{eq:Mtilde-def}
  \widetilde{M}(\delta,s) &\coloneqq \widetilde{M}_-\left(\delta,s\right) + \widetilde{M}_+\left(\delta, s\right), \\
  \widetilde{M}_{\pm}\left(\delta,s\right) &\coloneqq \min \left\{ M \in \N \; \big| \; G_{\pm}\left(M,s\right) \leq \frac{\delta}{2}\right\}.
\end{align}
For $\delta = 10^{-2}$, $10^{-4}$, and $10^{-6}$, we display the values of these $\widetilde{M}$ quantities in Figure \ref{fig:Mvals}. We see that for small values of $s$, the requisite number of points $\widetilde{M}$ scales like $1/s$. In particular, for small $s$, more effort (quadrature points) is allocated to $\widetilde{M}_+$, but for small $1-s$, comparatively more effort is allocated to $\widetilde{M}_-$.
\begin{figure}[htbp]
  \begin{center}
    \resizebox{\textwidth}{!}{
      \includegraphics{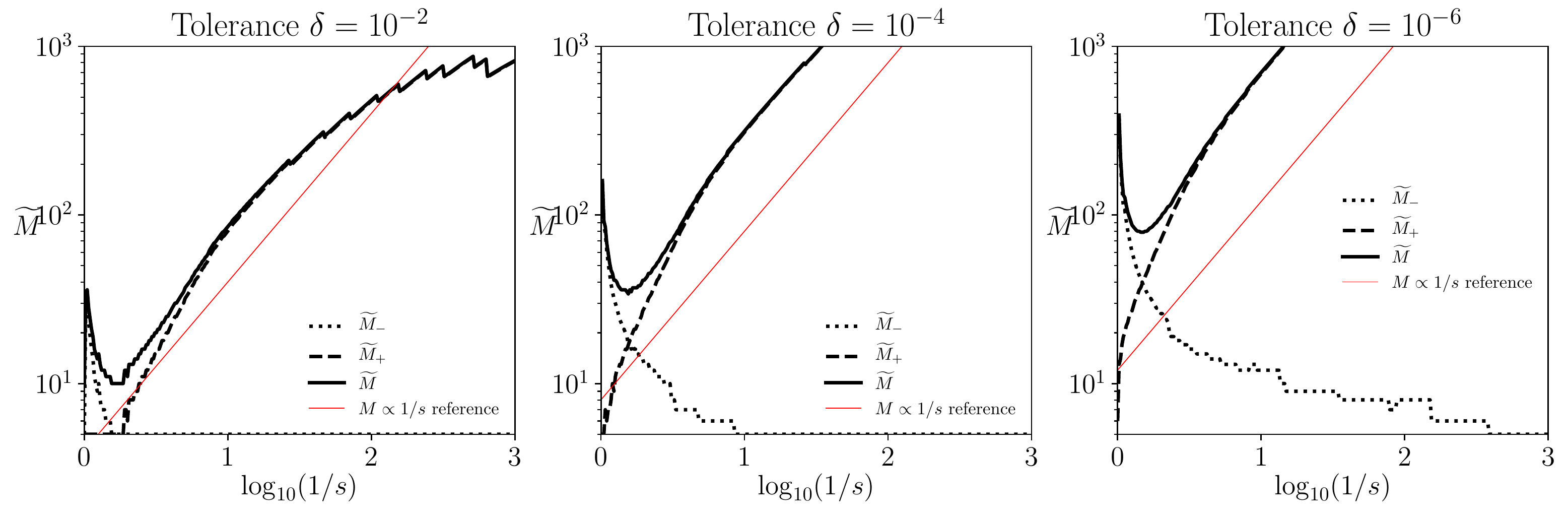}
    }
  \end{center}
  \caption{\label{fig:Mvals} Values of $\widetilde{M}$ and $\widetilde{M}_\pm$ defined in \eqref{eq:Mtilde-def} for various values of the tolerance $\delta$. We show $s$-dependence as $\log_{10}(1/s)$ since the error behavior for small $s$ is most restrictive. For visual reference, a $1/s$ curve is also plotted. We see that for small values of $s_{\pm}$, the corresponding value of $\widetilde{M}_{\pm}$ is large.}
\end{figure}
Therefore, the number of quadrature queries $M_- + M_+$ in the fully discrete scheme \eqref{eq:gq-truth-solution} can be quite large when the fractional order $s$ is very small. This general observation, including the $1/s$-type behavior shown in Figure \ref{fig:Mvals}, is consistent with earlier work \cite{antil_short_2017}. Thus, the number of classical PDE solutions $M = M_- + M_+$ needed to compute an accurate solution is large. This motivates a need to make these solves more efficient; we achieve this in the next section via model reduction.

\section{Model reduction for the integral formulation}\label{sec:rbm}
This section proposes an augmentation of the algorithm in the previous section. The cost of computing the fully discrete solution \eqref{eq:gq-truth-solution} is essentially $M_- + M_+$ queries of finite element solvers for $\bs{w}_{\pm}^\calN$. In practice one can require $M_- + M_+ \sim 100$, cf. Figure \ref{fig:Mvals} and earlier work \cite{bonito_numerical_2015, CWeiss_BvBWaanders_HAntil_2018a}, resulting in a substantial computational cost if the cost of computing $\bs{w}_{\pm}^\calN$ is high. 

We observe that the formulations \eqref{eq:wpm-continuous-def} for $w_{\pm}$ (and also \eqref{eq:wpm-discrete-def} for the discrete counterparts $\bs{w}_{\pm}^\calN$) are quintessential examples of parameterized PDE's where RBM algorithms are used to accelerate solution queries. Thus, RBM can be used to ameliorate the cost of performing $M_- + M_+$ queries of these PDE's. In RBM terminology, an available expensive discrete solution is called a \textit{truth} solution. Thus, our truth solutions for the auxiliary PDE problem for $w_{\pm}$ are $w^\calN_{\pm}$ defined in \eqref{eq:wpm-discrete-def}. The associated truth solution for $u(s)$ is \eqref{eq:gq-truth-solution}. The purpose of RBM procedures is to diminish the cost of evaluating the truth solution.

\subsection{Reduced basis methods}
Let $\mathcal{L}$ be a (classical) differential operator, and consider the following PDE parameterized by a Euclidean parameter $y \in D \subset \R^p$:
\begin{align}\label{eq:rbm-problem}
  \mathcal{L}(w; x; y) &= f(x;y), & (x,y) &\in \Omega \times D \subset \R^d \times \R^p,
\end{align}
where $x$ is the spatial variable and $y$ is a parameter. The operator $\mathcal{L}$ is differential in the $x$ variable. For example, the PDE defining $w(y)$ from \eqref{eq:aux-pde} can be written as \eqref{eq:rbm-problem} with the operator,
\begin{align*}
  \mathcal{L} &= -\alpha \Delta + \beta I, & p &= 2, & y = (\alpha, \beta) \in D &=  (0, 1]^2.
\end{align*}
We assume that \eqref{eq:rbm-problem} is well-posed for each $y \in D$. For a fixed $y$, one usually develops an $x$-discretization with $\calN \gg 1$ degrees of freedom yielding a solution $w^\calN$ with membership in an $\calN$-dimensional subspace. For us, this is the discretization defined in section \ref{ssec:space-discretization}. We assume that $\calN$ is large enough so that
\begin{align*}
  \sup_{y \in D} \| w^\calN(y) - w(y) \|_{L^2} \leq \epsilon,
\end{align*}
where $\epsilon$ is a user-prescribed tolerance. Thus, the map $y \mapsto w^\calN(y) \approx w(y)$ requires algorithms whose complexity is dependent on $\calN$.\footnote{For linear elliptic operators $\mathcal{L}$, this complexity can in principle scale like $\mathcal{O}\left(\calN \log \calN\right)$, but frequently is $\mathcal{O}(\calN^2)$, or even $\mathcal{O}(\calN^3)$ depending the details of the employed numerical solver.} Such an algorithm that performs the operation $y \mapsto w^\calN(y)$ is called a \textit{truth} approximation or solver.

The reduced basis method (RBM) is a thematic collection of model reduction strategies for parameterized PDEs that compute an emulator $y \mapsto w_N(y) \approx w^\calN(y)$, whose complexity behaves like $\mathcal{O}(N^3)$ or $\mathcal{O}(N^2)$, where $N \ll \calN$. For $N/\calN$ sufficiently small, this can result in an emulator $w_N$ whose evaluation is substantially cheaper than the truth approximation $w^\calN$. The RBM emulator takes the form,
\begin{align}\label{eq:rbm-def}
  w_N(y) &\coloneqq \sum_{k=1}^N c_{N,k}(y) \phi_k, & \phi_k &\in V_N \coloneqq \mathrm{span}\left\{ w^\calN(y_1), \ldots, w^\calN(y_N) \right\} \subset V,
\end{align}
where $\{y_k\}_{k=1}^N$ are particular parameter values that are chosen during the RBM construction procedure. The success of RBM algorithms rely on three main components:
\begin{itemize}
  \item The condition that the \textit{manifold} of solutions,
    \begin{align*}
      W(D) \coloneqq \left\{ w(y) \; \big|\; y \in D \right\} \subset L^2(\Omega),
    \end{align*}
    is ``low rank". The mathematically precise statement of this is that the Kolmogorov $N$-width of the manifold,
    \begin{align*}
      d_N(W) \coloneqq \inf_{\substack{V \in L^2 \\ \dim V = N}} \sup_{y \in D} \inf_{v \in V} \left\| w(y) - v \right\|_{L^2(\Omega)},
    \end{align*}
    decays quickly with $N$. ``Quickly" ideally means exponentially, but high algebraic rates of decay are also suitable. This condition ensures an RBM emulator $w_N$ can achieve $L^2$-proximity to the truth approximation $w^\calN$ when $N/\calN$ is very small. We provide empirical evidence in this paper that this condition is true, and show this rigorously in a follow-up paper \cite{antil_nwidth_2019}. 
  \item The condition that the truth approximation $w^\calN(y)$ comes with a practically computable \textit{a posteriori} error estimate $\Delta(y)$, satisfying,
    \begin{align*}
      \Delta(y) \gtrsim \| w(y) - w^\calN(y) \|_{L^2}.
    \end{align*}
    This usually comes in the form of \textit{a posteriori} finite element estimates, and in practice in the algorithm are actually used to measure $\|w^N(y) - w^\calN(y)\|$. This condition ensures that the parameter values $\{y_n\}_{n=1}^N$ in \eqref{eq:rbm-def} can be chosen in a computationally tractable manner. In our case the PDE's we consider are linear so that efficient residual-based error indicators $\Delta$ can be derived.
  \item The condition that the operator $\mathcal{L}$ and right-hand side $f$ have \textit{affine} dependence on the parameter $y$. This means that one has the expressions,
    \begin{align*}
      \mathcal{L} &= \sum_{q=1}^{Q_{\mathcal{L}}} \gamma_q(y) \mathcal{L}_q, & f(x;y) = \sum_{q=1}^{Q_f} \sigma_q(y) f_q(x),
    \end{align*}
    where we have introduced (i) $y$-independent differential operators $\mathcal{L}_q$, (ii) $y$-independent functions $f_q(x)$, (iii) $x$-independent functions $\gamma_q(y)$, and (iv) $x$-independent functions $\sigma_q(y)$. More precisely, one requires the weak (variational) form of $\mathcal{L}$ to have such a decomposition. This condition is needed so that evaluation of the RBM emulator map $y \mapsto u^N(y)$ can be accomplished using operations that are independent of the truth approximation discretization parameter $\calN$. We will briefly justify this for our situation in the next section. 
\end{itemize}
Our next goal is to apply the RBM algorithm to the truth discretizations of $w_{\pm}(y)$ that define the fractional solution $u(s)$. We discuss this in the next section.

\subsection{RBM formulation}
We describe here the RBM procedure for approximating $\bs{w}_-^\calN$ via a reduced basis emulator; the procedure for $\bs{w}_+^\calN$ is nearly identical. We recall the discrete truth approximation formulation that defines $\bs{w}_-^\calN$:
\begin{align}\label{eq:truth-w-}
  \left( \bs{S} + e^{-y} \bs{M}\right) \bs{w}^\calN_-(y) &= \bs{f}, 
\end{align}
where $\bs{S}$, $\bs{M}$, and $\bs{f}$ are defined in \eqref{eq:truth-discretization}. Detailed exposition of application of the RBM algorithm to this parameterized PDE (and to much more general cases) can already be found in existing textbook literature \cite{patera_reduced_2007,hesthaven_certified_2016,quarteroni_reduced_2016}. Here we give a only brief synopsis of the major steps in the algorithm for completeness, but refer to the previously-mentioned references for details and motivating explanation of the algorithm. In particular, in what follows we describe the algorithm using vectors and matrices instead of more common functional-analytic mathematical statements; this choice is made for simplicity of exposition since the algorithm itself is not new and we instead focus on the application of the algorithm.

As described in \eqref{eq:rbm-def}, the RBM method produces emulator $w_{n,-}$ defined as 
\begin{align}\label{eq:rbm-solution}
  w_{n,-}(y) &= \sum_{k=1}^n c_{n,k}(y) w_-^\calN(y_k), & \bs{c_n} &= \left( c_{n,1}, \ldots, c_{n,n} \right)^T,
\end{align}
where we have made a particular choice of the basis $\phi_k$ appearing in \eqref{eq:rbm-def}.\footnote{The parameter values $y_k$ and coefficients $c_{n,k}$ should be labeled $y_{k,-}$ and $c_{n,k,-}$, respectively, to differentiate them from the analogous quantities resulting from applying RBM to $w^\calN_+$. However, we omit this notational dependence for more clarity in exposition.}  Also, since the RBM procedure builds $w_N$ sequentially by first building $w_1, w_2, \ldots, $, we label the RBM dimension as $n$, satisfying $1 \leq n \leq N$, in this section. We must specify the parameter values $\{y_k\}_{k=1}^n$ and the coefficients $\{c_{n,k}\}_{k=1}^n$, which is the focus of the following discussion.

\subsection{Computing the $c_{n,k}$}
We assume that $y_1, \ldots, y_n$ have been chosen and are known, and now seek to define the coefficients $\{c_{n,k}\}_{k=1}^n$, equivalently the vector $\bs{c}_n$, whose computation allow evaluation of $y \mapsto w_n(y)$. To proceed we define a new matrix $\bs{U} \in \R^{\calN \times n}$, having entries
\begin{align*}
  \bs{U}_{n,-} &= \left[ \bs{w}^\calN(y_1) \hskip 5pt \bs{w}^\calN(y_2) \hskip 5pt \cdots \hskip 5pt \bs{w}^\calN(y_n) \right], & \left(U_{n,-}\right)_{j,k} &= w^\calN_j(y_k),
\end{align*}
where the vector $\bs{w}^\calN$ and its entries $w^\calN_j$ are expansion coefficients for the truth approximation solution, see \eqref{eq:truth-expansion}.

Then for each $y$, the coefficients $c_{n,k}$ of the RBM solution are defined by seeking the vector $\bs{c}_n(y) \in \R^n$ satisfying
\begin{align}\label{eq:rbm-condition}
  \bs{U}^T \left( \bs{S} + e^{-y} \bs{M} \right) \bs{U} \bs{c}_n(y) = \bs{U}^T \bs{f}.
\end{align}
Assuming $\bs{U}$ has linearly independent columns (which is assured by the choice of $y_k$ discussed in the next section), then this uniquely defines $\bs{c}_n(y)$ for each $y$, and prescribes the RBM solution $w_n$ via \eqref{eq:rbm-solution}. One final point of interest is that our truth variational form \eqref{eq:truth-w-} exhibits \textit{affine} dependence on the parameter $y$, making it possible to compute $\bs{c}_n$ very efficiently. We may rearrange computations in \eqref{eq:rbm-condition} so that 
\begin{align*}
  \left( \bs{B} + e^{-y} \bs{C} \right) \bs{c}(y) = \bs{g}, 
\end{align*}
where
\begin{align*}
  \bs{B} &\coloneqq \bs{U}^T \bs{S} \bs{U} \in \R^{n \times n}, & \bs{C} &\coloneqq \bs{U}^T \bs{M} \bs{U} \in \R^{n \times n}, & \bs{g} &\coloneqq \bs{U}^T \bs{f} \in \R^n,
\end{align*}
so that the quantities $\bs{B}$, $\bs{C}$, and $\bs{g}$, once computed, are all \textit{independent} of \textit{both} the truth discretization dimension $\calN$ and the parameter $y$. 
Thus, for each $y$, the coefficients $\bs{c}_n$ (i.e., the RBM solution $w_n$) can be computed with complexity that depends \textit{only} on $n$ and \textit{not} on $\calN$. 
Since in practice $n \ll \calN$ this can result in computational savings, especially if we wish to query $y \mapsto w^n_-(y)$ numerous times. 
This is one of the major attractions of model reduction with RBM. 

\subsection{Choosing $y_k$}
The ingredient we are left to provide to complete our description of the RBM algorithm is the choice of parameter values $y_k$ in \eqref{eq:rbm-solution}. Given an RBM approximation $w_n$, we focus on the choice of $y_{n+1}$. We accomplish this via the standard greedy procedure in RBM algorithms. Ideally, this choice is given by 
\begin{align}\label{eq:rbm-greedy}
  y_{n+1} = \argmax_{y \geq 0} \left\| w^\calN(y) - w_n(y) \right\|.
\end{align}
Unfortunately, this explicit form requires computing the full solution $w^\calN(y)$ at all parameter values $y$, which RBM seeks to avoid. To circumvent this restriction, the standard strategy is to resort to residual-based error indicators. The following lemma identifies one such computable residual-based error indicator $\Delta_n(y)$. 
\begin{lemma}
  Define the residual vector $\bs{r}_n(y) \in \R^{\calN}$ as 
  \begin{align}\label{eq:resdef}
    \begin{split}
    \bs{r}_{n,-}(y) &\coloneqq \bs{f} - \left( \bs{S} + e^{-y} \bs{M} \right) \bs{U}_{n,-} \bs{c}_{n,-}(y), \\
    \bs{r}_{n,+}(y) &\coloneqq \bs{f} - \left( \bs{S} e^{-y} + \bs{M} \right) \bs{U}_{n,+} \bs{c}_{n,-}(y),
    \end{split}
  \end{align}
  and the indicator 
  \begin{align}\label{eq:Delta-def}
    \begin{split}
    \Delta_{n,-}\left(y\right) &\coloneqq \frac{C_\calN^2}{\sqrt{\lambda_{\mathrm{min}}\left(\bs{M}\right)}\left( C_\calN^2 e^{-y} + 1 \right)} \|\bs{r}_n(y)\|, \\
    \Delta_{n,+}\left(y\right) &\coloneqq \frac{C_\calN^2}{\sqrt{\lambda_{\mathrm{min}}\left(\bs{M}\right)}\left( e^{-y} + C_\calN^2 \right)} \|\bs{r}_n(y)\|, 
    \end{split}
  \end{align}
  Then
  \begin{align}\label{eq:rbm-error-estimator}
    \Delta_{n,\pm}(y) \geq \left\| w^\calN_{\pm}(y) - w_{n,\pm}(y) \right\|
  \end{align}
\end{lemma}
\begin{proof}
  We show the result for the ``$-$" quantities; a similar proof works for the ``$+$" quantities. The residual $\bs{r}_{n,-}$ satisfies
  \begin{align*}
    \left( \bs{S} + e^{-y} \bs{M} \right) \left( \bs{w}^\calN_-(y) - \bs{w}_{n,-}(y) \right) = \bs{r}_{n,-},
  \end{align*}
  so that 
  \begin{align*}
    \left\| w^\calN_-(y) - w_{n,-}(y) \right\|_{L^2} = \left\| \bs{w}_-^\calN(y) - \bs{w}_{n,-}(y) \right\|_{\bs{M}} \leq \frac{1}{\lambda_{\mathrm{min}}\left(\bs{S} + e^{-y} \bs{M}, \bs{M} \right)} \left\| \bs{M}^{-1/2} \bs{r}_{n,-} \right\|,
  \end{align*}
  with
  \begin{align*}
    \lambda_{\mathrm{min}}\left(\bs{S} + e^{-y} \bs{M}, \bs{M} \right) \coloneqq \inf_{\bs{v} \in \R^{\calN}} \frac{ \bs{v}^T \left( \bs{S} + e^{-y} \bs{M} \right) \bs{v}}{\bs{v} \bs{M} \bs{v}} = e^{-y} + \inf_{\bs{v} \in \R^{\calN}} \frac{\bs{v}^T \bs{S} \bs{v}}{\bs{v}^T \bs{M} \bs{v}} = e^{-y} + \lambda_{\mathrm{min}}\left(\bs{S}, \bs{M}\right).
  \end{align*}
  To summarize, we have the estimate
  \begin{align*}
    \left\| w_-^\calN(y) - w_{n,-}(y) \right\| &\leq \frac{1}{e^{-y} + \lambda_{\mathrm{min}}}\left(\bs{S}, \bs{M}\right) \left\| \bs{M}^{-1/2} \bs{r}_{n,-} \right\| \\
                                                 &\leq \frac{1}{\sqrt{\lambda_{\mathrm{min}}\left(\bs{M}\right)}\left( e^{-y} + \lambda_{\mathrm{min}}\left(\bs{S}, \bs{M}\right)\right)} \|\bs{r}_{n,-}\|,
  \end{align*}
  which is the desired result by using the definition of $C_\calN$ in \eqref{eq:CN-def}.
\end{proof}
The residual vectors $\bs{r}_{n,\pm}$ can be efficiently computed for many values of $y$. We illustrate this $\bs{r}_{n,-}$. We have:
\begin{align}\nonumber
  \bs{r}_{n,-}(y) &\coloneqq \bs{f} - \left( \bs{S} + e^{-y} \bs{M} \right) \bs{U}_{n,-} \bs{c}_{n,-}(y) \\\label{eq:resminus}
                  &= P_{\mathcal{R}(\bs{R}_n)^\perp} \bs{f} + P_{\mathcal{R}(\bs{R}_n)} \left( \bs{f} - \left( \bs{S} + e^{-y} \bs{M} \right) \bs{U}_{n,-} \bs{c}_{n,-}(y) \right),
\end{align}
where $P_{\mathcal{R}(\bs{A})}: \R^\calN \rightarrow \R^\calN$ is the $\R^{\calN}$-orthogonal projector onto the column space of a matrix $\bs{A}$, and 
\begin{align*}
  \bs{R}_n \coloneqq \left[ \bs{S}\bs{U}_{n,-}, \hskip 3pt \bs{M}\bs{U}_{n,-} \right] \in \R^{\calN \times (2 n)}.
\end{align*}
The orthogonal decomposition in \eqref{eq:resminus} shows that the Pythagorean theorem can be used to compute the Euclidean vector norm $\|\bs{r}_{n,-}(y)\|_2$ in an efficient way for several values of $y$: 
\begin{itemize}
  \item $\|P_{\mathcal{R}(\bs{R}_n)^\perp} \bs{f}\|_2$ is $y$-independent, so that it can be computed once and stored.
  \item $P_{\mathcal{R}(\bs{R}_n)} \left( \bs{f} - \left( \bs{S} + e^{-y} \bs{M} \right) \bs{U}_{n,-} \bs{c}_{n,-}(y) \right)$ is a vector in a $2n$-dimensional, $y$-independent vector space. Thus, the norm can be computed with only $n$-dependent complexity. The fact that $\bs{c}_n(y)$ appears with linear behavior in this expression ensures that we can rearrange computations so that, for each $\bs{c}_n(y)$, the norm of this quantity can be computed using complexity that is dependent only on $n$.
\end{itemize}
In summary, while the right-hand side of \eqref{eq:rbm-error-estimator} is not efficiently computable for many values of $y$, the left-hand side is efficiently computable for several values of $y$ since $\|\bs{r}_n\|_2$ is efficient to compute, and $\lambda_{\mathrm{min}}\left(\bs{S}, \bs{M}\right)$ does not depend on $y$ and can be computed either directly or iteratively with generalized eigenvalue solvers once and subsequently stored.

Standard greedy algorithms for RBM methods require a computable quantity satisfying \eqref{eq:rbm-error-estimator}, and replace the essentially un-computable maximization \eqref{eq:rbm-greedy} with the computable maximization
\begin{align}\label{eq:rbm-weak-greedy}
  y_{n+1} = \argmax_{y \geq 0} \Delta_n(y).
\end{align}
The above maximization has an objective function that is efficiently computable, and the inequality \eqref{eq:rbm-error-estimator} ensures that the maximization \eqref{eq:rbm-weak-greedy} is a \textit{weak} greedy algorithm. Weak greedy algorithms in turn ensure that the set of chosen parameters $\{y_1, \ldots, y_N\}$ defines an RBM subspace $V_N$ in \eqref{eq:rbm-def} whose best approximation to the truth solution $\bs{w}^\calN_-$ is comparable to the Kolmogorov $N$-width \cite{binev_convergence_2011,devore_greedy_2013}. 

We have completed the basic description of the RBM algorithm: the emulators $w_{N,\pm}(y)$ are defined by computing coefficients $\bs{c}_{N,\pm}(y)$ as described in the previous section, and the parameter values $y_k$ are chosen according to \eqref{eq:rbm-weak-greedy} by computing the estimators $\Delta_{n,\pm}(y)$ for $n = 1, 2, \ldots, $.  One usually sequentially computes $y_k$ until $\sup_{y \geq 0} \Delta_{n,\pm}(y)$ is smaller than some specified tolerance, so that one can rigorously certify the error committed by the RBM emulators. This tolerance condition is usually how the terminal RBM dimension $N$ is computationally specified.

One final observation we make is a major theoretical result of this paper: 
\begin{theorem}\label{thm:u-certificate}
  Let $N_{\pm}$ denote the RBM dimensions formed for the emulators $w_{N_{\pm},\pm}(y)$. Then 
  \begin{align}\label{eq:rbm-u-certificate}
    \sup_{s \in (0,1)} \left\| u^\calN(s) - u_{N}(s) \right\| \leq \frac{4}{\pi} \max \left\{ \sup_{y \geq 0} \Delta_{N_+,+}(y), \sup_{y \geq 0} \Delta_{N_-,-}(y) \right\}.
  \end{align}
\end{theorem}
\begin{proof}
  We have
  \begin{align*}
    u^\calN(s) - u_N(s) &= \beta_0(s_-) \sum_{j=1}^{N_-} \tau_{j,-} \left[ w^\calN_-\left(\frac{y_{j,-}}{s_-}\right) - w_{N_-,-} \left(\frac{y_{j,-}}{s_-}\right) \right] + \\ 
     & \beta_0(s_+) \sum_{j=1}^{N_+} \tau_{j,+} \left[ w^\calN_+\left(\frac{y_{j,+}}{s_+}\right) - w_{N_+,+} \left(\frac{y_{j,+}}{s_+}\right) \right] .
  \end{align*}
  Taking $L^2$ norms of both sides, and using the triangle inequality with \eqref{eq:rbm-error-estimator} yields the result.
\end{proof}
We note that the quantities $\Delta_{N_\pm,\pm}(y)$ are computed during the RBM construction phase, so that these estimators are available. Therefore, \eqref{eq:rbm-u-certificate} provides a computable error bound that can be used to certify error committed by using the RBM algorithm.

\begin{remark}\label{rem:general-b}
  Just as with Remark \ref{rem:general-b}, all our results above extend to a more general elliptic operator $\mathcal{E}$ satisfying the assumptions outlined in Remark \ref{rem:general-b}. In this case, all of our formulas in this section carry over.
\end{remark}

\subsection{Algorithm summary}\label{ssec:rbm:algorithm}
The full algorithm of this section first uses the RBM algorithm to perform model reduction on the parameterized PDE solutions $w^\calN_{\pm}(y)$.\footnote{In the previous sections we have describe this process only for $w^\calN_-$, but the process for $w^\calN_+$ results in almost the same procedure with only minor differences stemming from the location of the $e^{-y}$ factor.}
Subsequently, the efficient RBM emulators $w_{N,\pm}(y)$ are used in the GQ algorithm from Section \ref{ssec:gq:algorithm}. We describe the full algorithm in Algorithm \ref{alg:RBM}. 

\begin{algorithm}
  \caption{RBM algorithm: Produces solution to the fractional Laplace problem \eqref{eq:fpde}. The function \textsc{OfflineFracLapRBM} needs to be completed only once and has complexity dependent on $\mathcal{N} = \dim V$. Afterwards, \textsc{OnlineFracLapRBM} can be called for arbitrary values of $s \in (0, 1)$ with a computational cost dependent only on $\dim V_N = N \ll \mathcal{N}$.
    \label{alg:RBM}}
  \begin{algorithmic}[1]
    \Require{Availability of a discrete solution $\bs{w}^{\calN}(\alpha,\beta)$ from the formulation \eqref{eq:w-galerkin}.}
    \Function{OfflineFracLapRBM}{}
      \State $n \gets 1$. Randomly choose $y_1$, compute and store $w^\calN_-(y_1)$. 
      \State Assemble RBM emulator $w_1(\cdot)$.
      \For{$n \gets 2 \textrm{ to } N_-$}
        \State Compute $y_{n}$ from \eqref{eq:rbm-weak-greedy}.
        \State Compute and store $w^\calN_-(y_n)$.
        \State Assemble RBM emulator $w_n(\cdot)$.
      \EndFor
      \State Repeat above computations to assemble $w_{N_+,+}(\cdot)$
      \State \Return{RBM emulators $w_{N_{\pm},\pm}(\cdot)$}
    \EndFunction
    \Statex
    \Require{Availability of RBM emulators $w_{N_{\pm},\pm}(\cdot)$.}
    \Function{OnlineFracLapRBM}{$s$}
    \State Call \textsc{FracLapGQ}$(s)$, Algorithm \ref{alg:GQ}, replacing $w^\calN_{\pm}(\cdot)$ with RBM emulators $w_{N_{\pm},\pm}(\cdot)$.
    \State \Return{$\bs{u}_N(s)$.}
    \EndFunction
  \end{algorithmic}
\end{algorithm}
The algorithm we have described in this section is a skeleton version of a modern RBM algorithm. We summarize various improvements that should be implemented in order for an RBM algorithm to be efficient and accurate:
\begin{itemize}
  \item One does not usually solve \eqref{eq:rbm-weak-greedy} by maximizing over the parameter continuum, and instead maximizes over a discrete set. For bounded parameter domains, it is common to use a uniform grid and subsequently adaptively (e.g., dyadically) refine the grid to ensure that no local maxima are skipped. Over the unbounded domain, we employ a logarithmic map; see section \ref{ssec:results-rbm-offline} for details.
  \item Our RBM ansatz \eqref{eq:rbm-solution} uses solution \textit{snapshots} $w^\calN_{\pm}(y)$ as basis functions. This is known to generally lead to ill-conditioning in the formulation \eqref{eq:rbm-condition} even for small $n$ (since in practice the columns of $\bs{U}_{n,\pm}$ are ``nearly" linearly dependent). A better prescription is to build the RBM basis functions by orthogonalizing snapshots.
  \item Some \naive{} implementations of the decomposition \eqref{eq:resminus} via quadratic forms leads to numerical roundoff error that results in stagnation of the error indicators $\Delta_{n,\pm}(y)$ near root-machine precision. (E.g., for double precision, stagnation occurs when $\Delta_{n,\pm}(y)$ takes values around $10^{-8}$. More careful computations allow one to overcome this limitation \cite{casenave_accurate_2012,casenave_accurate_2014,buhr_numerically_2014,chen_robust_2019}. 
\end{itemize}
We again refer to \cite{patera_reduced_2007,hesthaven_certified_2016,quarteroni_reduced_2016} for a more complete description of important but standard RBM algorithm details.

Finally, we note that one can combine the error estimates in \eqref{eq:rbm-u-certificate} and \eqref{eq:quad-error} via the triangle inequality to create a computable bound for $\| u^N - \widetilde{u}^\calN \|$. This error would estimate the error committed by the two novel innovations of this paper: our Gaussian quadrature approach and the model reduction procedure.

\section{Numerical examples}\label{sec:results}

In these experiments, we compare the effectiveness of our improvements to current methods. We first describe the setup of the test problems that we consider in our simulations. We solve \eqref{eq:fpde} on $\Omega = [0,1]^2$ with homogeneous Dirichlet boundary conditions. We test a total of 3 algorithms:
\begin{itemize}
  \item ``SQ" --- The sinc quadrature approach from \cite{bonito_numerical_2015}. 
  \item ``GQ" --- Algorithm \ref{alg:GQ} detailed in section \ref{sec:gq}, utilizing a modified version of the integral formulation in \cite{bonito_numerical_2015} along with Gaussian quadrature.
  \item ``RBM" --- The approach detailed in section \ref{ssec:rbm:algorithm}, leveraging the reduced basis method to accelerate the GQ algorithm.
\end{itemize}

To compare the three methods, we will use the same number of $y$-quadrature points $M = M_- + M_+$ in each approach. 
This number represents the total number of classical PDE solves needed to compute an approximation to $u(s)$.
In particular, we make the choices given in~\cite{bonito_numerical_2015,antil_short_2017}, which depend on the spatial mesh and fractional order:
\begin{align}\label{eq:N-quad}
  M_+ &= \bigg\lceil \frac{\pi^2}{4sk^2} \bigg\rceil, & M_- &= \bigg\lceil \frac{\pi^2}{4(1-s)k^2} \bigg\rceil, & k&=\frac{1}{\log(\sqrt{\calN})}
\end{align}
The finite element discretization is accomplished with linear quadrilateral finite elements on a Cartesian tessellation of $\Omega$. The one-dimensional grids that define this Cartesian tessellation are isotropic with respect to the two dimensions, and are defined as equidistant meshes with $2^{K}$ points. We will use various values of $K$. 

\subsection{Manufactured Solutions on $[0,1]^2$}
Consider the physical domain $\Omega = [0, 1]^2$. In this case, an explicit family of eigenfunctions for the Laplacian with homogeneous Dirichlet boundary conditions is available:
\begin{align*}
  \phi_{n,m}(x) &= \sin(n \pi x_1) \sin(m \pi x_2), & n, m &\in \N, & x &= (x_1, x_2) \in \Omega, 
\end{align*}
which satisfy
\begin{align*}
  -\Delta \phi_{n,m}(x) &= \lambda_{n,m} \phi_{n,m}(x), & \lambda_{n,m} = \pi^2 (n^2 + m^2).
\end{align*}
With this in hand, and using the inverse of the relation \eqref{eq:fraclap-explicit}, we can easily construct explicit solutions for testing using eigenfunction expansions. We explore the effectiveness of our algorithms through three manufactured solutions:
\begin{itemize}
  \item ``\textsc{Sine}" --- The function $u$ and data $f$ are, in this case,
    \begin{align*}
      u(x) &= \frac{1}{(2 \pi^2)^s} \sin(\pi x_1) \sin(\pi x_2), & f(x) &= \sin( \pi x_1) \sin(\pi x_2) 
    \end{align*}
  \item ``\textsc{Mixed modes}" --- The function $u$ and data $f$ are, in this case,
    \begin{align*}
      u(x) &= \frac{1}{(116 \pi^2)^s} \sin(4 \pi x_1) \sin(10 \pi x_2), & f(x) &= \sin( 4 \pi x_1) \sin( 10\pi x_2) 
    \end{align*}
  \item ``\textsc{Square bump}" --- The data $f$ is the indicator function
    \begin{align*}
      f(x) = \mathbbm{1}_{[0.25, 0.75]^2}(x).
    \end{align*}
    While an analytical solution is available as an infinite sum of eigenfunctions, we instead numerically compute a solution to the above problem on a highly refined mesh and consider this the ``exact" solution.
\end{itemize}

\subsection{Spatial convergence}
Our first test verifies that we recover spatial convergence in terms of the finite element mesh size. Since are primarily interested in accuracy and not efficiency, this section compares the \textsc{SQ} and \textsc{GQ} methods. 

We can see that the proposed \textsc{GQ} algorithm performs slightly better than the existing \textsc{SQ} algorithm for the same number of quadrature points $M$. We also remark that the \textsc{GQ} implementation allows us to generate solutions for small fractional parameters with less numerical difficulty. With the \textsc{SQ} approach, the difficulty arises when application of the quadrature rule results in large values of a term involving $e^{y}$ appearing in operators that must be inverted.

\begin{figure}[htbp]
\begin{center}
  \resizebox{\textwidth}{!}{
 \includegraphics[width=0.32\textwidth]{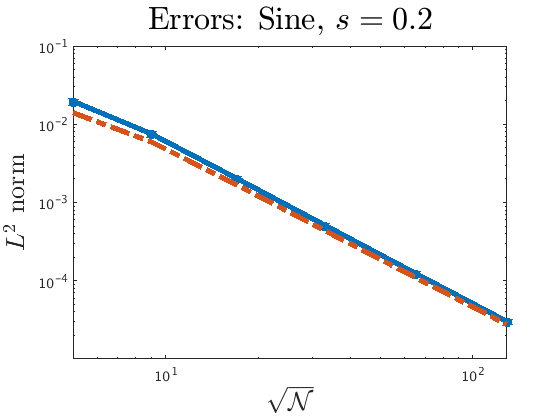} 
 \includegraphics[width=0.32\textwidth]{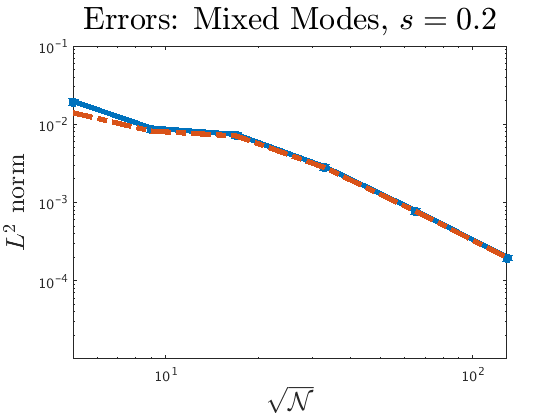}
 \includegraphics[width=0.32\textwidth]{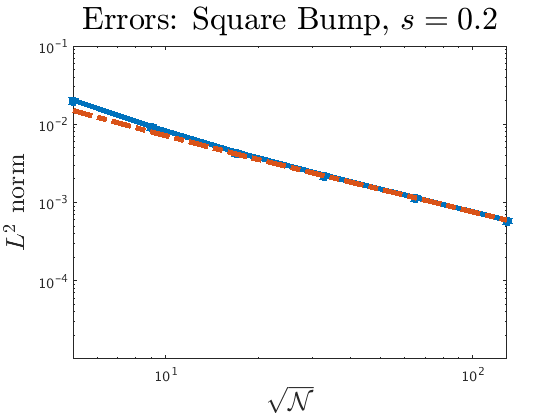}
 }
  \resizebox{\textwidth}{!}{
 \includegraphics[width=0.32\textwidth]{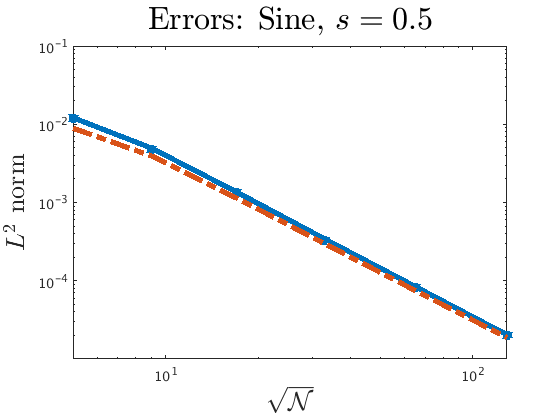} 
 \includegraphics[width=0.32\textwidth]{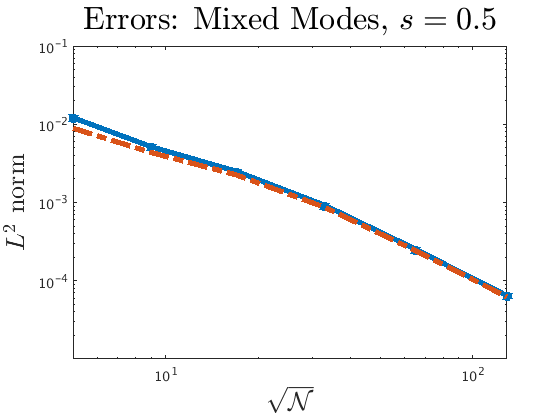}
 \includegraphics[width=0.32\textwidth]{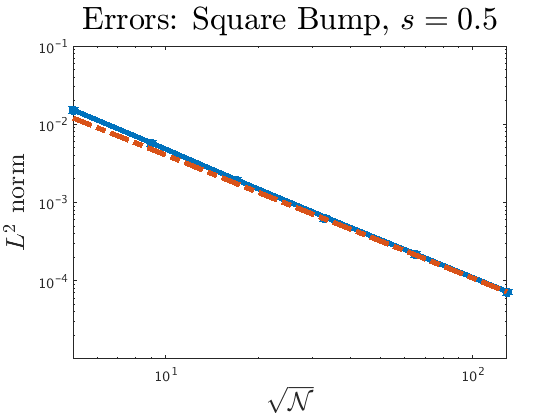}
 }
  \resizebox{\textwidth}{!}{
 \includegraphics[width=0.32\textwidth]{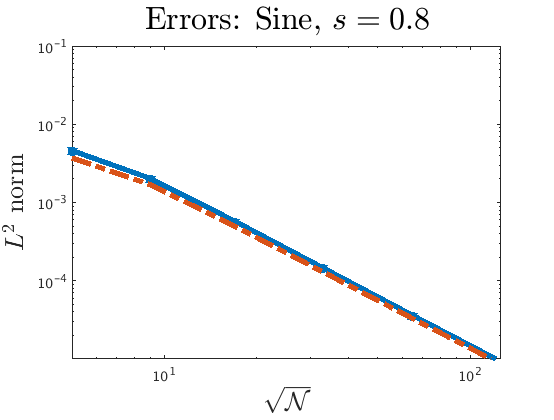} 
 \includegraphics[width=0.32\textwidth]{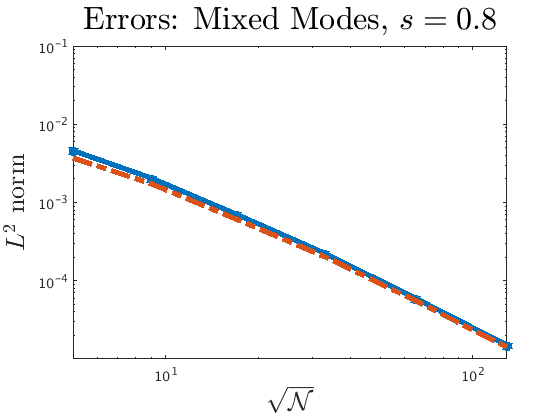}
 \includegraphics[width=0.32\textwidth]{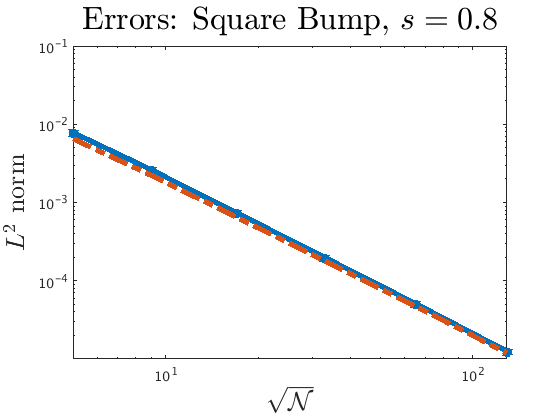}
 }
\end{center}
  \caption{Convergence of SQ (solid blue) and GQ methods (red dashed) as the spatial mesh is refined. 
  Each used a dyadic mesh along the spatial variable with increasing resolution. 
  A parameter value of $s=0.2$ was used and similar results were seen for value of $s$ between $0.1$ and $0.9$. We used the number of quadrature points for the integral suggested by current methods ~\cite{antil_short_2017}.
 \label{fig:h-convergence}}
\end{figure}

\subsection{Quadrature rule efficiency}
In this section we compute errors committed by the \textsc{GQ} and \textsc{SQ} algorithms for different values of the quadrature rule size $M$. The purpose of this test is to understand the efficiency of the quadrature rule, i.e., the number of solutions of $w^\calN_{\pm}$ required. Figure \ref{fig:quadplot} illustrates errors for the three test cases as a function of the total number of quadrature nodes. The test in this section does not fix $M_\pm$ as given by \eqref{eq:N-quad}. Instead, given a number of quadrature points $M$ (the abscissa in Figure \ref{fig:quadplot}), we generate $M$-point quadrature rules for the Gaussian quadrature and sinc approaches. Thus, the quadrature rule for each $M$ is generated anew.

The results indicate that the \textsc{GQ} algorithm converges faster than the \textsc{SQ} with respect to the number of PDE solves. This shows that the GQ algorithm appears to be far more efficient than sinc quadrature for computing solutions to these fractional problems.
\begin{figure}[htbp]
\begin{center}
  \resizebox{\textwidth}{!}{
 \includegraphics[width=0.32\textwidth]{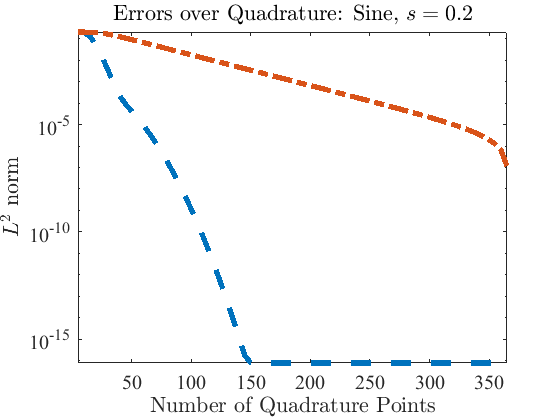}
 \includegraphics[width=0.32\textwidth]{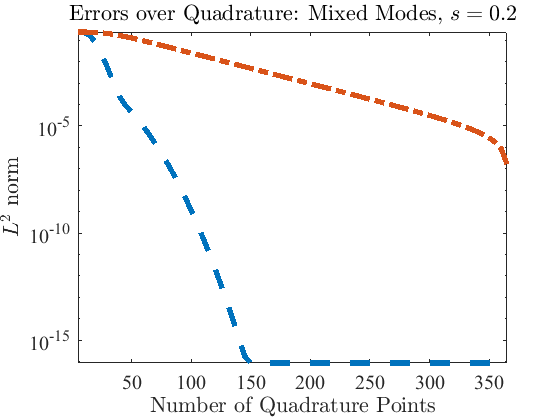}
 \includegraphics[width=0.32\textwidth]{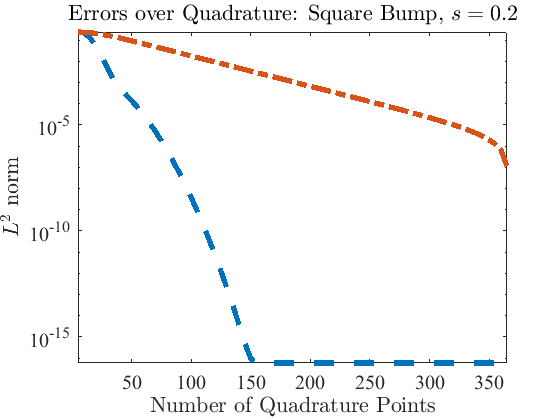} 
 }
  \resizebox{\textwidth}{!}{
 \includegraphics[width=0.32\textwidth]{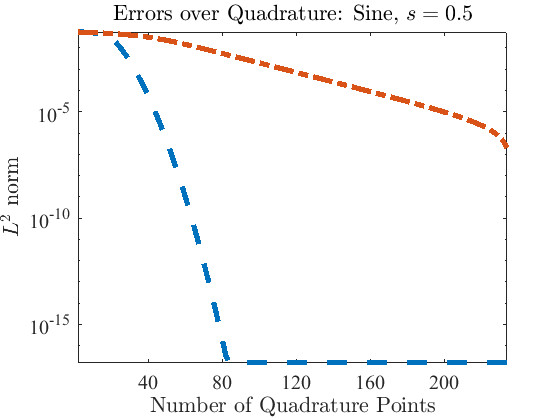}
 \includegraphics[width=0.32\textwidth]{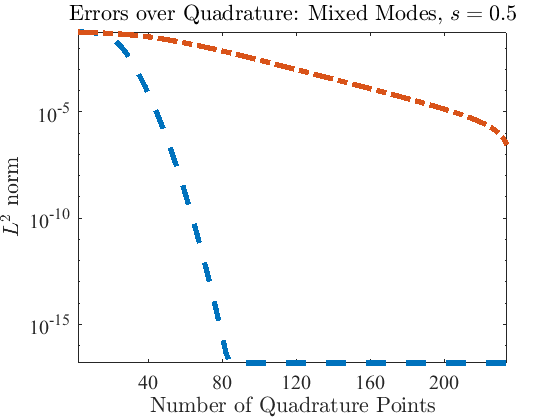}
 \includegraphics[width=0.32\textwidth]{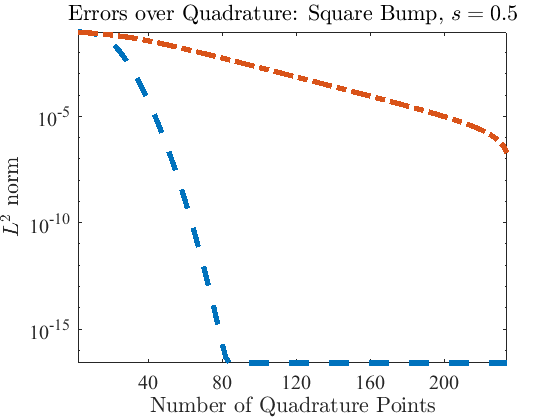} 
 }
  \caption{\label{fig:quadplot}Accuracy comparison of the SQ (red, dotted and dashed) and GQ methods (blue, dashed), with fractional order $s=0.2$ (top plots) and $s = 0.5$ (bottom plots). Similar results where observed for value of $s$ between $0.1$ and $0.9$.}
\end{center}
\end{figure}
\subsection{RBM offline efficiency}\label{ssec:results-rbm-offline}
This section investigates the \textsc{RBM} algorithm. For now, we restrict our attention to one-time queries of $u(s)$, i.e., to situations when, given $\Omega$, $f$, and $s$, we seek to compute \textit{only} $u(s)$ for this given $s$. For the \textsc{GQ} algorithm this involves a single run of the routine \textsc{FracLapGQ} in Algorithm \ref{alg:GQ}. For the \textsc{RBM} algorithm, this entails a single run of the \textsc{OfflineFracLapRBM} routine in Algorithm \ref{alg:RBM}, followed by a single run of \textsc{OnlineFracLapRB} routine. 

For the \textsc{RBM} algorithm, we solve \eqref{eq:rbm-weak-greedy} by discretizing the $y \in [0, \infty)$ domain in a uniform way under a logarithmic map. Precisely: we set $z = e^{-y}$ for $y \in [0, \infty)$ and proceed to discretize $z \in [0,1]$. We take 128 equispaced points in the $z$ variable and map back to $y$-space with $z \mapsto -\log z = y$. We subsequently perform a discrete maximization over this set instead of the continuous optimization \eqref{eq:rbm-weak-greedy}. 

In figure \ref{fig:offline-time} we compare the \textsc{GQ} algorithm to the accelerated \textsc{RBM} algorithm, including the offline construction time. We see that the initial investment of the RBM algorithm in the offline phase is substantial, accumulating to the time required for the direct \textsc{GQ} method with 150-200 quadrature points. However, we see that after this initial offline investment, subsequent evaluations of the RBM surrogate are extremely efficient, so that the effort required to evaluate $M \gg 1$ quadrature point is essentially the same as that required to evaluate at a single quadrature point.

\begin{figure}[H]
\begin{center}
  \resizebox{\textwidth}{!}{
 \includegraphics[width=0.32\textwidth]{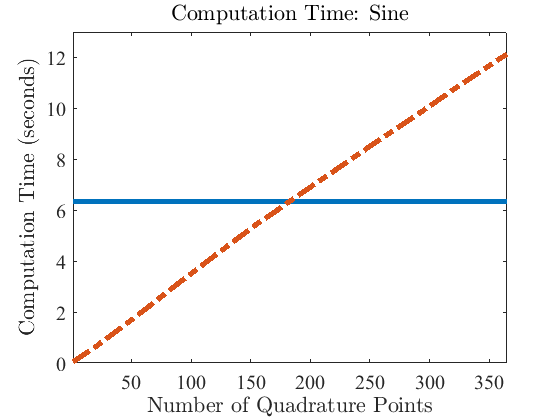}
 \includegraphics[width=0.32\textwidth]{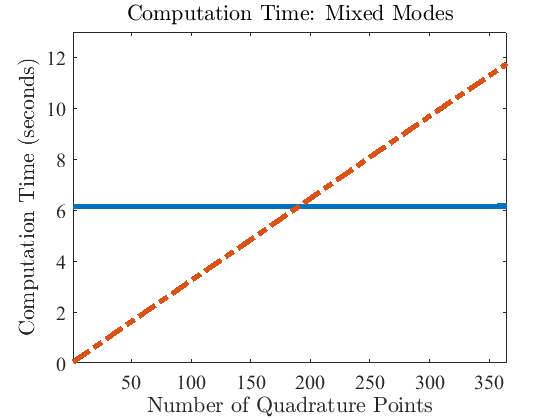}
 \includegraphics[width=0.32\textwidth]{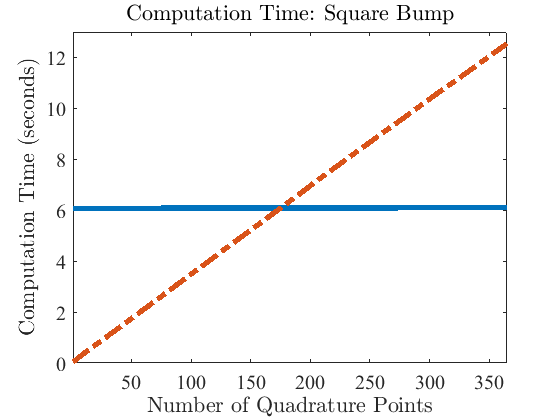} 
 }
 \caption{\label{fig:offline-time}``Offline" (i.e., one-time) computational investment for a single solve of \eqref{eq:fpde} with a fixed value of $s = 0.2$. These experiments compare both the direct (dotted and dashed) and reduced basis methods (solid) using the gaussian quadrature. Each used a dyadic mesh with $7$ levels. Similar results where seen for value of $s$ between $0.1$ and $0.9$. A dyadic spatial mesh and its corresponding number of quadrature points were used. 
}
\end{center}
\end{figure}

\subsection{RBM accuracy}
We now investigate the accuracy delivered by the RBM algorithm in the construction of reduced order models for $w^\calN_\pm$. Our rigorous error certificate for $u(s)$ using the reduced order model is \eqref{eq:rbm-u-certificate}, but we here consider a finer estimate using the proof of Theorem \ref{thm:u-certificate}. We define the error estimator,
\begin{align*}
  \Delta_N(s) \coloneqq \sum_{\sigma \in \{+,-\}} \beta_0(s_\sigma) \sum_{j=1}^{M_\sigma} \Delta_{N,\sigma}(s_{j,\sigma}) \tau_{j,\sigma},
\end{align*}
where $\Delta_{N,\pm}$ are the computable error indicators defined in \eqref{eq:Delta-def}, and we again choose $M_{\pm}$ as in \eqref{eq:N-quad}. One can see from the proof of Theorem \ref{thm:u-certificate} that this quantity bounds the error committed by the RBM procedure. We show the values of this quantity in Figure \ref{fig:Delta} as a function of $N$, and observe that it decays exponentially.

Finally, we remark that $\Delta_N$ only certifies the error committed by the model reduction RBM algorithm; the error committed by the $y$-quadrature rule is not certified by this quantity.

\begin{figure}[htbp]
\begin{center}
 \includegraphics[width=0.32\textwidth]{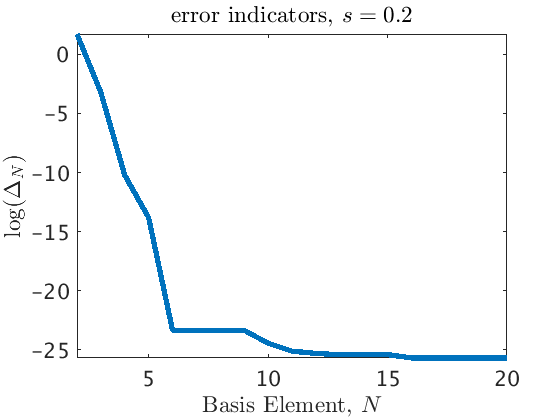} 
 \includegraphics[width=0.32\textwidth]{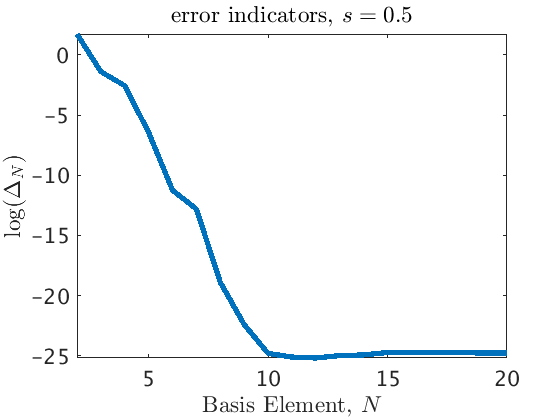}
\includegraphics[width=0.32\textwidth]{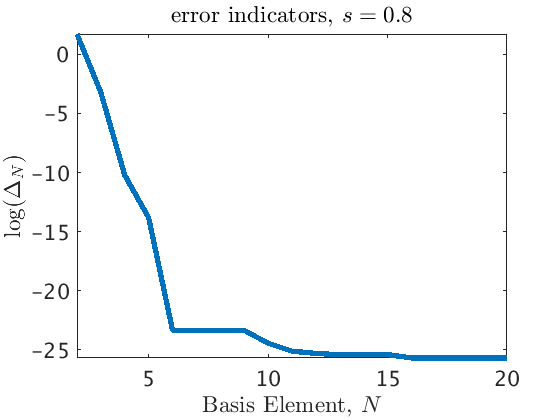}
\end{center}
\caption{\label{fig:Delta}Error indicators $\Delta_N(s)$ as a function of $N$, for $s = 0.2, 0.5, 0.8$.}
\end{figure}

\subsection{RBM accuracy and efficiency}
Finally, we explore the accuracy afforded by the \textsc{RBM} procedure for computing $s \mapsto u(s)$, and also verify the computational efficiency of the procedure. In Figure \ref{fig:timing} left and center, we demonstrate that the error committed by the RBM algorithm is stable, even for relatively small values of the parameter $0.01 < s < 0.1$. Furthermore, the right pane of this figure demonstrates that if we wish to \textit{repeatedly} query the map $s \mapsto u(s)$ for several values of $s$, the \textsc{RBM} algorithm is undeniably more efficient by an order of magnitude even for just one query, and by three orders of magnitude if 1000 queries are needed.

\begin{figure}[htbp]
\begin{center}
 \includegraphics[width=0.32\textwidth]{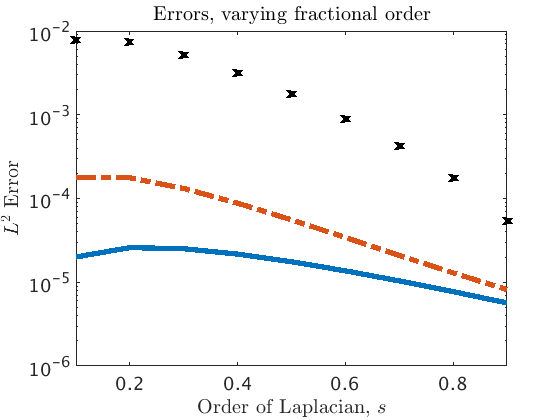} 
 \includegraphics[width=0.32\textwidth]{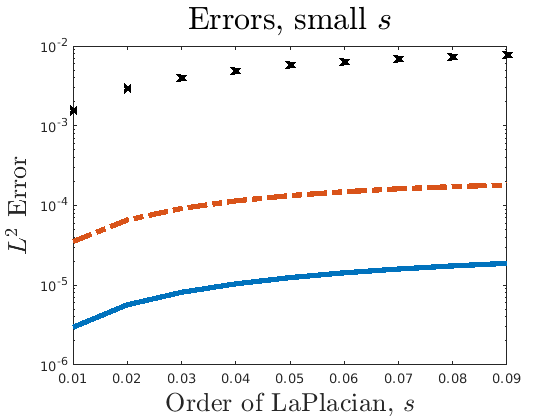}
\includegraphics[width=0.32\textwidth]{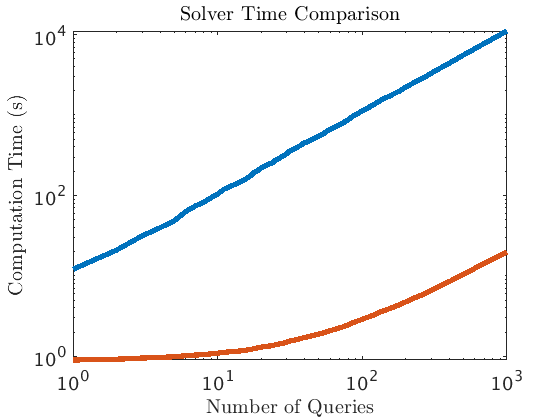}
\end{center}
\caption{\label{fig:timing}Accuracy of the \textsc{RBM} algorithm over a range of values of $s$ (left and center). The \textsc{Sine} example is plotted in a red dot-dashed, the \textsc{Mixed Modes} in a blue solid line, and the \textsc{Square Bump} case in black crosses. In the right pane we show the \textit{cumulative} computational time required by the \textsc{GQ} algorithm (blue) versus the \textsc{RBM} algorithm (red). Each query refers to an evaluation of the map $s \mapsto u(s)$. In particular this cumulative time for the \textsc{RBM} solver includes the one-time offline cost required by \textsc{OfflineFracLapRBM} in Algorithm \ref{alg:RBM}.}
\end{figure}


\section{Conclusion}
We propose a novel model reduction strategy for computing solutions to fractional Laplace PDE's, in particular \eqref{eq:fpde}. Our algorithm builds on the ideas introduced in \cite{bonito_numerical_2015}, improving accuracy and stability, and accelerating that algorithm considerably. Our model reduction strategy hinges on the fact that the solution to the fractional problem can be written in terms of classical, local elliptic PDE's, for which RBM-based model reduction is known to be efficient. 

We provide novel stability bounds for both the continuous and discrete problems, and our numerical experiments suggest that our Gaussian quadrature approach is more efficient than alternative quadrature methods. All of our algorithmic and theoretical results apply to solutions to differential equations involving fractional powers of general elliptic operators. A rigorous proof of the convergence for our quadrature rule is the subject of ongoing study.

\bibliography{dtRbmBibHD,references} 
\bibliographystyle{plain}

\appendix 

\end{document}